\documentclass[11pt, oneside]{amsart}
\usepackage{amsmath}
\usepackage{amsthm}
\usepackage{amssymb} %mathrsfs}
\usepackage[mathscr]{euscript}
\usepackage{amsmath}
\usepackage{amsthm}
\usepackage{amssymb}
\usepackage{fancyhdr}
\usepackage{enumerate}
\usepackage{amscd}
\usepackage[all]{xy}
\usepackage{graphicx}
\usepackage{color}
\usepackage{hyperref}
\hypersetup{
    colorlinks,
    citecolor=black,
    filecolor=black,
    linkcolor=black,
    urlcolor=black
}
\usepackage{tikz}
\usepackage[all]{xy}
\usepackage{graphicx}
\usetikzlibrary{backgrounds}
\theoremstyle{plain}
\newtheorem{lemma}{Lemma}[section]

\newtheorem{proposition}[lemma]{Proposition}
\newtheorem{thm}[lemma]{Theorem}
\newtheorem{theorem}[lemma]{Theorem}
\newtheorem{Theorem}{Theorem}
\newtheorem{corollary}[lemma]{Corollary}
\newtheorem{Corollary}[Theorem]{Corollary}
\newtheorem*{surjective lemma}{Lemma~\ref{lemma:surjective}}

\numberwithin{equation}{section}

\theoremstyle{remark}

\theoremstyle{definition}
\newtheorem{remark}[lemma]{Remark}

\newtheorem{definition}[lemma]{Definition}
\newtheorem{Definition}[Theorem]{Definition}

\newtheorem*{ack}{Acknowledgements}

\DeclareMathOperator{\Div}{div}

\newcommand{\too}{\longrightarrow}

\usepackage[tmargin=0.96in,bmargin=1in,lmargin=1in,rmargin=1in]{geometry}
\begin{document}
\title[Equality in the Spacetime Positive Mass Theorem]{Equality in the Spacetime Positive Mass Theorem}
\author{Lan-Hsuan Huang}
\address{Department of Mathematics, University of Connecticut, Storrs, CT 06269, USA}
\email{lan-hsuan.huang@uconn.edu}
\author{Dan A. Lee}
\address{CUNY Graduate Center and Queens College}
\email{dan.lee@qc.cuny.edu}
\thanks{The first author was partially supported by the NSF Career Award DMS~1452477, National Center for Theoretical Sciences (NCTS) in Taiwan, Simons Fellowship of the Simons Foundation, and   von Neumann Fellowship at the Institute for Advanced Study.}

\begin{abstract}
We affirm the rigidity conjecture of the spacetime positive mass theorem in dimensions less than eight. Namely, if an asymptotically flat initial data set satisfies the dominant energy condition and has $E=|P|$,  then  $E=|P|=0$,  where $(E, P)$ is the ADM energy-momentum vector. The dimensional restriction can be removed if we assume the positive mass inequality holds.  Previously the result was only known for spin manifolds~\cite{Beig-Chrusciel:1996, Chrusciel-Maerten:2006}.
%We consider a variational approach to the Regge-Teitelboim Hamiltonian, except that we use a modified constraint operator introduced by the first named author and J.~Corvino~\cite{Corvino-Huang:2016} in place of the usual constraint operator. The spacetime positive mass inequality  implies that an initial data set satisfying $E=|P|$ must locally minimize the modified Regge-Teitelboim Hamiltonian among initial data sets with the same modified constraints. The Lagrange multipliers corresponding to this constrained minimizer give rise to asymptotically vacuum Killing initial data which is also asymptotically translational.  Earlier work of R.~Beig and P.~Chru\'{s}ciel  \cite{Beig-Chrusciel:1996} then implies that the ADM energy-momentum vector must be zero.  Since the variational formalism takes place in the space of initial data sets of low regularity, we prove a spacetime positive mass inequality $E\ge |P|$ in the setting of weighted Sobolev spaces, which may be of independent interest. 
\end{abstract}

\maketitle

\section{Introduction}\label{section:introduction}

Our main result is the following theorem that affirms the rigidity conjecture of the spacetime positive mass theorem (see \cite[p. 398]{Witten:1981}, also \cite[p. 84]{Eichmair-Huang-Lee-Schoen:2016} and the references therein). We refer to Section~\ref{section:prelim} for precise statements of terms used below.

\begin{Theorem}\label{theorem:main1}
Let $3\le n \le 7$. Let $(M, g, k)$ be an $n$-dimensional asymptotically flat initial data set that satisfies  the dominant energy condition and has $E=|P|$, where $(E, P)$ is the ADM energy-momentum vector. Then $E=|P|=0$.
\end{Theorem}

We emphasize that our proof only uses the positive mass inequality (proven in \cite{Eichmair-Huang-Lee-Schoen:2016} for $3\le n \le 7$) as an input  and does not use its proof in any way, and thus our result holds in arbitrary dimensions whenever the positive mass inequality holds. We describe our generalization of Theorem~\ref{theorem:main1} more precisely as follows. See Definition~\ref{definition:AF} for the precise definition of asymptotic flatness used in this paper, as well as our hypotheses on $(p, q, q_0, \alpha)$.

\begin{Definition}\label{definition:inequality}
Let $(M, g, k)$ be an asymptotically flat initial data set of type $(p, q, q_0, \alpha)$. We say that \emph{the positive mass inequality holds near $(g, k)$} if there is an open ball  centered at $(g, k)$ in $C^{2,\alpha}_{-q}\times C^{1,\alpha}_{-1-q}$ such that for each asymptotically flat initial data set $(\bar{g}, \bar{k})$ in that open ball of the same type satisfying the dominant energy condition, we have $\bar{E} \ge |\bar{P}|$,  where $(\bar{E}, \bar{P})$ is the ADM energy-momentum vector of $(\bar{g}, \bar{k})$.   
\end{Definition}

\begin{Theorem}\label{theorem:main}
Let $n\ge 3$.  Let $(M, g, k)$ be an $n$-dimensional asymptotically flat initial data set with the dominant energy condition. Suppose that the positive mass inequality holds near $(g, k)$.  If $E=|P|$, then  $E=|P|=0$.
\end{Theorem}

The above statement was proved in three dimensions by R.~Beig and P.~Chru\'{s}ciel  using the spinor approach in 1996 ~\cite{Beig-Chrusciel:1996}, and  has been directly extended  by Chru\'{s}ciel and D.~Maerten for \emph{spin} manifolds in higher dimensions \cite{Chrusciel-Maerten:2006}.  Our proof of Theorem~\ref{theorem:main} is a different, variational approach that applies generally without the spin assumption. 

We give a brief history of the positive mass theorem. The special case  $k=0$ is often called the Riemannian positive mass theorem. In this case, $|P|=0$ and the dominant energy condition is reduced to the condition that the scalar curvature of $g$ is nonnegative everywhere. R.~Schoen and S.-T.~Yau  proved  the Riemannian positive mass theorem  $E\ge 0$ in dimension less than eight using minimal surfaces \cite{Schoen-Yau:1979-pmt1} (see also \cite{Schoen-Yau:1981-asymptotics, Schoen-Yau:1979-pat, Schoen:1989}). 
%They extended the approach to dimensions less than eight  via an induction argument on dimension \cite{Schoen-Yau:1979-pat, Schoen:1989}. 
In higher dimensions, the induction argument may break down due to possible singularities of minimal hypersurfaces. Recently, Schoen and Yau proved the Riemannian positive mass theorem in all dimensions \cite{Schoen-Yau:2017}. Since the proof of the inequality $E\ge 0$ is by contradiction, a separate argument is used to give a characterization of the  equality case that if $E=0$, then $(M, g)$ is isometric to Euclidean space.

In the case $k\ne 0$, Schoen and Yau also proved that $E\ge 0$ in dimension three using the Jang equation to reduce to the Riemannian case  \cite{Schoen-Yau:1981-pmt2}. 
% $E\ge 0$ theorem is also sometimes called the positive mass theorem in the literature, but we would like to refer it more accurately as \emph{the  positive energy theorem} and refer to the more general  $E\ge |P|$ theorem as the positive mass theorem. 
M.~Eichmair generalized the Jang equation argument and proved the $E\ge 0$  theorem in dimensions less than eight \cite{Eichmair:2013}. These results also show that if $E=0$, then $(M, g, k)$ can be isometrically embedded in Minkowski spacetime with the second fundamental form $k$.  %This characterization follows from combining the Jang equation argument with the rigidity of the Riemannian positive mass theorem. 

Together with Eichmair and Schoen, the authors proved that the  positive mass inequality $E\ge |P|$ holds in dimensions less than eight ~\cite{Eichmair-Huang-Lee-Schoen:2016} by using marginally outer trapped hypersurfaces (MOTS) in place of the minimal hypersurfaces used in the Schoen-Yau proof of the Riemannian positive mass theorem. Since MOTS are not known to obey a useful variational principle, a major part of the proof is to find an appropriate substitute of the first variational formula for the area functional that can be used to produce the MOTS-stability. The dimensional restriction is due to possible singularities of MOTS, just as in the Riemannian case.  We note that it was previously understood that a heuristic ``boost argument'' shows  that the $E\ge 0$ theorem \emph{implies} the positive mass inequality. In that same paper, we also made rigorous the heuristic boost argument reduction by proving a new density theorem.  Using  the boost argument, J.~Lohkamp has announced a new compactification argument to prove positivity for $n \ge 3$ in~\cite{Lohkamp:2016}. We note that both the MOTS approach and  the boost argument are by contradiction, so they do not give any information about the equality case  $E=|P|$, which is addressed in the current paper.

There is a different approach to the positive mass theorem due to Witten \cite{Witten:1981} (see also \cite{Parker-Taubes:1982}). The proof can be extended to \emph{spin} manifolds of all \hbox{dimensions~\cite{Ding:2008, Bartnik:1986}.} In his paper, Witten also gave a sketch to characterize the  $E=|P|$ case for vacuum initial data sets, which led to the conjecture that the only possibility for $E=|P|$ is when $E=|P|=0$ and $(M,g,k)$ embeds as a slice of Minkowski space. The conjecture in dimension three under various stronger assumptions  was proved by A.~Ashtekar and G.~Horowitz~\cite{Ashtekar-Horowitz:1982} and  P.F.~Yip~\cite{Yip:1987}. As mentioned above, a complete and rigorous proof is due to Beig and Chru\'{s}ciel in three dimensions~\cite{Beig-Chrusciel:1996} and Chru\'{s}ciel and Maerten for spin manifolds in higher dimensions~\cite{Chrusciel-Maerten:2006}.

%There is a different approach to the positive mass theorem due to E.~Witten.  In 1981 he used a spinor argument to show $E\ge |P|$ for all $3$-manifolds \cite{Witten:1981} (see also \cite{Parker-Taubes:1982}) and to characterize the case $E=0$. The proof can be extended to \emph{spin} manifolds of all \hbox{dimensions~\cite{Ding:2008, Bartnik:1986}.} In his paper, Witten also gave a sketch to characterize the  $E=|P|$ case for vacuum and electrovacuum spacetimes, which led to the conjecture that the only possibility for $E=|P|$ is when $E=|P|=0$ and $(M,g,k)$ embeds as a slice of Minkowski space. The conjecture in dimension three under various stronger assumptions  was proved by A.~Ashtekar and G.~Horowitz~\cite{Ashtekar-Horowitz:1982} and  P.F.~Yip~\cite{Yip:1987}. As mentioned above, a complete and rigorous proof is due to Beig and Chru\'{s}ciel in three dimensions~\cite{Beig-Chrusciel:1996} and Chru\'{s}ciel and Maerten for spin manifolds in higher dimensions~\cite{Chrusciel-Maerten:2006}. 

%Using the characterization of $E=0$ due to Schoen and Yau~\cite{Schoen-Yau:1981-pmt2} for $n=3$ and to Eichmair~\cite{Eichmair:2013}~for $3\le n \le 7$, under a stronger decay rate $\mathrm{tr}_g k = O(|x|^{-\gamma})$ for some $\gamma>2$, our main theorem implies the following characterization of the $E=|P|$ case. 

Combined with the aforementioned work of Schoen and Yau~\cite{Schoen-Yau:1981-pmt2} and Eichmair~\cite{Eichmair:2013} characterizing the $E=0$ case, our main theorem immediately implies the following. 
\begin{Corollary}
Let $3\le n \le  7$, and let $(M, g, k)$ be an $n$-dimensional asymptotically flat initial data set satisfying the dominant energy condition. If $n=3$, further assume\footnote{Note that for $n=3$, our asymptotic flatness assumption (Definition 2.5) only gives $s>3/2$.}  that $\mathrm{tr}_g k = O(|x|^{-s})$ for some $s>2$. If $E=|P|$, then $(M, g, k)$ can be isometrically embedded into Minkowski spacetime with the induced second fundamental form~$k$.
\end{Corollary}

We now outline the proof of Theorem~\ref{theorem:main}. Let $(M, g, k)$ be an asymptotically flat initial data set satisfying the dominant energy condition,  as well as the assumption $E=|P|$. Given a scalar function $f_0$ and a vector field $X_0$, we introduce a functional~$\mathcal{H}$ (see Definition~\ref{definition:functional}) on the space of initial data sets. The functional is obtained from the classical Regge-Teitelboim Hamiltonian  by replacing the usual constraint operator with the \emph{modified constraint operator} $\overline{\Phi}_{(g,\pi)}$ introduced by the first named author and J.~Corvino~\cite{Corvino-Huang:2019}. Choosing the pair  $(f_0, X_0)$ asymptoting to $(E, -2P)$, we apply the Sobolev positive mass inequality (Theorem~\ref{theorem:inequality}) to see that $(g, k)$ locally minimizes the functional $\mathcal{H}$ among initial data sets with the dominant energy condition. In contrast, the classical Regge-Teitelboim Hamiltonian is not known to have a local minimizer among the analogous constrained minimization. Using the theory of Lagrange multipliers, we produce a pair $(f, X)$ in the kernel of the adjoint $(D\overline{\Phi}_{(g,\pi)})^*$ of the linearized modified constraint operator that is asymptotic to $(f_0, X_0)$. Analyzing the solution to the equations $(D\overline{\Phi}_{(g,\pi)})^*(f, X)=0$, we obtain $E=|P|=0$.  

Our approach is motivated by the work of R.~Bartnik~\cite{Bartnik:2005} toward his quasi-local mass program. Aside from analytical technicalities, Bartnik's argument could be applied, under the additional assumption that $(g,k)$ is a three-dimensional vacuum initial data set. Using the new modified functional, we are able to handle general initial data sets with dominant energy condition. We also use a different analytical framework.

The paper is organized as follows. In Section~\ref{section:prelim}, we present the basic definitions and recall the modified constraint operator of~\cite{Corvino-Huang:2019}.  In Section~\ref{section:DEC}, we present a fundamental property of the modified constraint operator.  In Section~\ref{section:inequality}, we prove a Sobolev version of positive mass inequality. We also include a deformation result to the strict dominant energy condition (Theorem~\ref{theorem:strict}), which may be of independent interest. The main argument to prove Theorem~\ref{theorem:main} is in Section~\ref{section:variational}. %In Appendix~\ref{section:KID}, we re-prove a result of Beig-Chru\'{s}ciel~\cite{Beig-Chrusciel:1996} that completes the proof of our main theorem. We discuss regularity of solutions of the adjoint linearized modified constraint equations in Appendix~\ref{appendix:regularity}, and finally, we state and prove the version of the Lagrange multipliers theorem used in this paper in Appendix~\ref{section:Lagrange}.

\begin{ack}
The  authors would like to express their sincere gratitude to Richard Schoen for discussion and  support. They are also grateful to Hugh Bray, Justin Corvino, Greg Galloway, Jim Isenberg,  Christina Sormani, and Mu-Tao Wang for their encouragement. 
\end{ack}

\section{Preliminaries}\label{section:prelim}
\begin{definition}
Let $n\ge3$. An \emph{initial data set} is an $n$-dimensional smooth  manifold $M$ equipped with a weakly twice differentiable and continuous Riemannian metric $g$ and a weakly differentiable symmetric $(2, 0)$-tensor $\pi$ called the \emph{momentum tensor}. The momentum tensor is related to the more traditional $(0,2)$-tensor $k$, mentioned in Section~\ref{section:introduction}, via the equation
\[
	\pi^{ij}= k^{ij} -(\mathrm{tr}_g k)g^{ij},
\]
where the indices on the right have been raised using $g$.
The momentum tensor contains the same information as $k$ since $k^{ij} = \pi^{ij} - \frac{1}{n-1}(\mathrm{tr}_g \pi)g^{ij}$. 

We define the \emph{mass density} $\mu$ and the \emph{current density} $J$ (which is a vector quantity) by
\begin{align*} 
\begin{split}
	\mu&=\tfrac{1}{2}\left(R_g  + \tfrac{1}{n-1}(\mathrm{tr}_g \pi )^2 - |\pi|_g^2 \right)\\
	J&=  \Div_g \pi,
\end{split}
\end{align*}
where $R_g$ is the scalar curvature of $g$ and  $(\Div_g \pi)^i = \sum_j \pi^{ij}_{;j}$.  We define the \emph{constraint operator} on initial data by
\begin{align}\label{equation:constraint-map}
\Phi(g,\pi) = (2\mu, J) =\left (R_g  + \tfrac{1}{n-1}(\mathrm{tr}_g \pi )^2 - |\pi|_g^2, \Div_g \pi \right).
\end{align}
We say that $(M, g, \pi)$ satisfies the \emph{dominant energy condition} if 
\[ 
	\mu\ge |J|_g
\] 
everywhere in $M$. 
\end{definition}
We note that our definition of the constraint operator follows the preceding paper on the positive mass inequality~\cite{Eichmair-Huang-Lee-Schoen:2016}, but it causes discrepancies with the analogous formulas  in other references (e.g. \cite{Beig-Chrusciel:1996}) because of different normalizing conventions.

\begin{definition}
Let $B\subset \mathbb{R}^n$ be the closed unit ball  centered at the origin. For each nonnegative integer $k$, $ \alpha\in[ 0, 1]$, and $q\in \mathbb{R}$, we define the \emph{weighted H\"older space} $C^{k,\alpha}_{-q}(\mathbb{R}^n \setminus B)$ as the collection of those $f\in C^{k,\alpha}_{\mathrm{loc}}(\mathbb{R}^n\setminus B)$ with
\begin{align*}
	\| f \|_{C^{k,\alpha}_{-q}(\mathbb{R}^n \setminus B)}&:=\sum_{|I|\le k}\sup_{x\in \mathbb{R}^n\setminus B} \left| |x|^{|I|+q} (\partial^I f)(x)\right| + \sum_{|I|=k} \sup_{\substack{x, y\in \mathbb{R}^n\setminus B\\0<|x-y| \le |x|/2}}  |x|^{\alpha+|I|+q} \frac{| \partial^I f(x) - \partial^I f(y)|}{|x-y|^\alpha}<\infty.
\end{align*}
Let $M$ be a smooth manifold such that there is a compact subset $K\subset M$ and a diffeomorphism $M\setminus K \cong \mathbb{R}^n\setminus B$. We can define the $C^{k,\alpha}_{-q}$ norm on $M$ using an atlas of $M$ that consists of the diffeomorphism $M\setminus K \cong \mathbb{R}^n\setminus B$ and finitely many precompact charts, and then sum  the $C^{k,\alpha}_{-q}$ norm on the non-compact chart and the $C^{k,\alpha}$ norm on the precompact charts. We denote by $C^{k,\alpha}_{-q}(M)$ the completion of compactly supported  smooth functions with respect to the $C^{k,\alpha}_{-q}$ norm.  We use the notation $f= O^{k,\alpha} (|x|^{-q})$ interchangeably with  $f\in C^{k,\alpha}_{-q} (M)$.
\end{definition}

\begin{definition}\label{definition:weighted}
For each nonnegative integer $k$, $1\le p <\infty$, and $q\in \mathbb{R}$, we define the \emph{weighted Sobolev  space} $W^{k,p}_{-q}(\mathbb{R}^n\setminus B)$ as the collection of those $f$ with 
\[
	\| f \|_{W^{k,p}_{-q}(\mathbb{R}^n\setminus B)} := \left( \int_{\mathbb{R}^n\setminus B} \sum_{|I|\le k} \left| |x|^{|I|+q}  (\partial^I f)(x)\right|^p |x|^{-n} \, dx\right)^{1/p} < \infty.
\]
Suppose $M$ is a smooth manifold  such that there is a compact subset $K\subset M$ and a diffeomorphism $M\setminus K \cong \mathbb{R}^n\setminus B$. We can define the space $W^{k,p}_{-q}(M)$  as we did for $C^{k,\alpha}_{-q}(M)$ in the previous definition. We write $L^p_{-q}(M)$ instead of $W^{0,p}_{-q}(M)$.

We usually write $C^{k,\alpha}_{-q}$ for $C_{-q}^{k, \alpha}(M)$ and $W^{k,p}_{-q}$ for $W^{k,p}_{-q}(M)$ when the context is clear. The above norms can be extended to the tensor bundles of $M$ by summing the respective norms of the tensor components with respect to those charts.  It should be clear from context when we use the notation $C^{k,\alpha}_{-q}$ or $W^{k,p}_{-q}$  to denote spaces of functions or spaces of tensors.

\end{definition}
\begin{remark}\label{remark:inclusion}
Note that the above weighted spaces have a natural inclusion relation  $C^{k,\alpha}_{-q-\epsilon} \subset W^{k,p}_{-q}$ for any $\epsilon>0$.  On the other hand, by Sobolev embedding, if $p>n$, then $W^{k,p}_{-q} \subset C^{k-1, 1-\frac{n}{p} }_{-q}$. 
\end{remark}

\begin{definition}\label{definition:AF}
We assume
\[
	 n\ge 3, \qquad p>n,  \qquad q \in ( \tfrac{n-2}{2}, n-2), \qquad q_0 >0, \qquad \mbox{and} \qquad \alpha\in (0,1)
\] 
and, in addition,  
\begin{align}\label{equation:extra}
	q+\alpha >n-2.
\end{align}
Let $M$ be a connected smooth manifold without boundary. We say that an initial data set $(M, g, \pi)$ is \emph{asymptotically flat}  if  there is a compact subset $K \subset M$ and a diffeomorphism $M\setminus K  \cong \mathbb{R}^n \setminus B$ such that  
\begin{align}\label{equation:AF}
	(g - g_{\mathbb{E}}, \pi) \in \left(C^{2,\alpha}_{-q} \times C^{1,\alpha}_{-1-q} \right) \cap \left(W^{2,p}_{-q}\times W^{1,p}_{-1-q} \right)
\end{align}
and
\[
	\mu, J \in C^{0,\alpha}_{-n-q_0}
\]
where $g_{\mathbb{E}}$ is a complete smooth Riemannian background metric on $M$ that is equal to the Euclidean inner product in the coordinate chart $M\setminus K  \cong \mathbb{R}^n \setminus B$. We may sometimes refer to an asymptotically flat initial data set $(M, g, \pi)$ as being  \emph{of type $(p, q, q_0, \alpha)$} when we wish to emphasize the regularity assumption. 
\end{definition}

By the natural inclusion relation between H\"older and Sobolev spaces mentioned in Remark~\ref{remark:inclusion},  it suffices to assume $(g-g_{\mathbb{E}}, \pi)\in C^{2,\alpha}_{-q-\epsilon} \times C^{1,\alpha}_{-1-q-\epsilon}$ for some $\epsilon>0$, in place of  \eqref{equation:AF}. The current definition is for the convenience of fixing the fall-off rates of both H\"older and Sobolev spaces.

\begin{remark}
The assumption~\eqref{equation:extra} is not standard, but we include it as part of our definition because it is needed for technical reasons. It is only used in Theorem~\ref{theorem:kernel} (more specifically, Lemma~\ref{lemma:strong-decay}) and not elsewhere. %More specifically, in the proof of  Theorem~\ref{theorem:zero}, we apply techniques of Beig and Chru\'sciel to handle the lapse-shift pair obtained from our variational argument. The analogous theorem statements in~\cite[Theorem 3.4]{Beig-Chrusciel:1996} and~\cite[Theorem 2.5]{Chrusciel-Maerten:2006} assume the slightly higher regularity $(g, \pi)\in C^{3,\alpha}_{-q} \times C^{2, \alpha}_{-1-q}$ with $q>n-3$ and $\mu, J \in C^{1,\alpha}_{-n-q_0}$.
\end{remark}
%We note that our main result (Theorem~\ref{theorem:main}) holds for asymptotically flat initial data set of type $(q, q_0, \alpha)$  under an extra assumption that $q+\alpha >n-2$. 

\begin{remark}
Our main result still holds if we allow the above definition of initial data sets to have multiple asymptotically flat ends. We simply choose $(f_0, X_0)$ in the modified Regge-Teitelboim Hamiltonian (Definition~\ref{definition:functional}) to be identically zero on  other ends in the proof of Theorem~\ref{theorem:nontrivial-kernel}.
\end{remark}

\begin{definition}

The ADM energy $E$ and the ADM linear momentum $P=(P_1, \dots, P_n)$ of an asymptotically flat initial data set (named after Arnowitt, Deser, and Misner~\cite{ADM:1961}) are defined as 
\begin{align*}
	E&= \frac{1}{2(n-1)\omega_{n-1}} \lim_{r\to \infty} \int_{|x|=r}\sum_{i,j=1}^n (g_{ij,i}-g_{ii,j})\nu^j \, d\mathcal{H}^{n-1}\\
	P_i &= \frac{1}{(n-1)\omega_{n-1}} \lim_{r\to \infty} \int_{|x|=r} \sum_{i,j=1}^n \pi_{ij} \nu^j \, d\mathcal{H}^{n-1}
\end{align*}	
where the integrals are computed in $M\setminus K \cong \mathbb{R}^n \setminus B$, $\nu^j = x^j/|x|$, $d\mathcal{H}^{n-1}$ is the $(n-1)$-dimensional Euclidean Hausdorff measure,  $\omega_{n-1}$ is the volume of the standard $(n-1)$-dimensional unit  sphere, and the commas denote partial differentiation in the coordinate directions.  We  sometimes write the dependence on $(g, \pi)$ explicitly as $E(g,\pi)$ and $P(g,\pi)$.
\end{definition}

We now recall the modified constraint operator that was introduced by the first named author and J.~Corvino in \cite{Corvino-Huang:2019}, based on earlier study of the modified linearization in \cite[Section 6.1]{Eichmair-Huang-Lee-Schoen:2016}.

\begin{definition}
Given an initial data set $(M, g, \pi)$, we define the  \emph{modified constraint map $\overline{\Phi}_{(g,\pi)}$ at $(g,\pi)$} to be the operator on other initial data $(\gamma, \tau)$ given by
\begin{equation} \label{eq:def-mco} 
 \overline{\Phi}_{(g, \pi)}(\gamma, \tau)= \Phi(\gamma, \tau)+  \left(0, \tfrac{1}{2} \gamma \cdot (\textup{div}_g \pi)\right),
\end{equation}
where in local coordinates $ (\gamma \cdot (\textup{div}_g \pi))^i= g^{ij} \gamma_{jk} (\textup{div}_g \pi)^k$  and $\Phi(\gamma,\tau) $ is the usual constraint  \eqref{equation:constraint-map}. Here and throughout the paper, we use the Einstein summation convention.
\end{definition}

We denote its linearization at $(g, \pi)$  by $D\overline{\Phi}_{(g,\pi)}|_{(g,\pi)} $, or simply $D\overline{\Phi}_{(g,\pi)}$ for ease of notation. For a symmetric $(0,2)$-tensor $h$ and a symmetric $(2,0)$-tensor $w$, we have
\begin{align} \label{equation:modified}
	&D{\overline{\Phi}}_{(g,\pi)}(h , w) =D\Phi |_{(g,\pi)} (h,w) + (0,\tfrac{1}{2} h\cdot J)
\end{align}
where $J=\textup{div}_g \pi $ and 
\begin{align}\label{equation:constraint}
\begin{split}
	D\Phi|_{(g,\pi)} (h , w) &=\Big(L_g h -2 h_{ij} \pi_\ell^i \pi^{j\ell} - 2 \pi^j_k w^k_j  +\tfrac{2}{n-1}\mbox{tr}_g \pi (h_{ij} \pi^{ij} + \mbox{tr}_g w), \\
	&\qquad ( \mbox{div}_g w)^i - \tfrac{1}{2} \pi^{jk} h_{jk;\ell} g^{\ell i} + \pi^{jk} h^i_{j;k} +\tfrac{1}{2} \pi^{ij} (\mbox{tr}_g h)_{,j}\Big). 
\end{split}
\end{align}
Here all indices are raised or lowered using $g$, $L_g h:= -\Delta_g(\mbox{tr}_g h) + \mbox{div}_g \mbox{div}_g (h) - h^{ij} R_{ij}$ where $R_{ij}$ is the Ricci curvature of $g$,  and the semi-colon indicates covariant derivatives with respect to $g$. The formal adjoint operator of  $D\overline{\Phi}_{(g,\pi)}$ with respect to the $L^2$ product defined by $g$ has the following expression, for a function $f$ and a vector field $X$:
\begin{align} \label{equation:modified-adjoint}
	(D\overline{\Phi}_{(g,\pi)})^*(f, X) = D\Phi|^*_{(g,\pi)} (f, X) + \left( \tfrac{1}{2} X\odot J, 0\right),
\end{align}
where  $(X\odot J)_{ij}  = \frac{1}{2} (X_i J_j + X_j J_i)$ denotes the symmetric product,  and  $D\Phi|^*_{(g,\pi)} (f, X)$ is the adjoint operator of the  usual constraint map. Explicitly, 
\begin{align}
 	D\Phi|_{(g,\pi)} ^*(f, X)  & = \left(  L_g^*f +\left( \tfrac{2}{n-1} (\mbox{tr}_g \pi) \pi_{ij} - 2 \pi_{ik} \pi^k_j \right) f\right.
	\label{equation:adjoint}\\
	& \quad+ \tfrac{1}{2} \left( g_{i\ell}g_{jm} (L_X\pi)^{\ell m} + (\Div_g X) \pi_{ij}- J_i X_j - J_j X_i- X_{k;m} \pi^{km} g_{ij} - g(X, J) g_{ij} \right), \nonumber\\
	&\quad \left. -\tfrac{1}{2} (L_Xg)^{ij} + \left(\tfrac{2}{n-1} (\mbox{tr}_g \pi ) g^{ij}- 2 \pi^{ij}  \right) f\right) \nonumber
 \end{align}
where  $L_g^*f = -(\Delta_g f)g + \textup{Hess}_g f - f \textup{Ric}(g)$. The above formulas can be found in, for example, \cite[Lemma 2.3]{Corvino-Schoen:2006} for $n=3$, and \cite[Lemma 20]{Eichmair-Huang-Lee-Schoen:2016} and \cite[Section 2.1]{Corvino-Huang:2019} for general $n$.

Define  $\mathscr{M}^{2,p}_{-q}$ to be the set of symmetric $(0,2)$-tensors $\gamma$ such that  $\gamma-g_{\mathbb{E}}\in W^{2,p}_{-q}(M)$ and $\gamma$ is positive definite at each point. Note that by Sobolev embedding, $\gamma$ must be continuous (in fact, $C^{1,\alpha}_{\mathrm{loc}}$). That is, $\mathscr{M}^{2,p}_{-q}$ is the set of continuous Riemannian metrics that are asymptotic to $g_{\mathbb{E}}$ in $W^{2,p}_{-q}(M)$. Using an affine identification, note that we may regard $\mathscr{M}^{2,p}_{-q}$ as an open subset of the Banach space of  $W^{2,p}_{-q}$ symmetric $(0,2)$-tensors. 

We conclude the section with the following statement, whose proof is included in Appendix~\ref{section:surjective} for completeness.

\begin{lemma}[Cf. {\cite[Proposition 3.1]{Corvino-Schoen:2006},\cite[Lemma 20]{Eichmair-Huang-Lee-Schoen:2016}}] \label{lemma:surjective}
Let $(M, g, \pi)$ be an initial data set with $(g- g_\mathbb{E}, \pi) \in C^{2}_{-q} \times C^{1}_{-1-q}$. The modified constraint map $\overline{\Phi}_{(g,\pi)}:  \mathscr{M}^{2,p}_{-q}\times W^{1,p}_{-1-q}\too L^p_{-2-q}$ is smooth, and $D\overline{\Phi}_{(g,\pi)}: W^{2,p}_{-q}\times W^{1,p}_{-1-q}\too L^p_{-2-q}$ is surjective.  %Thus, $\overline{\Phi}_{(g,\pi)}$ is locally surjective.% in the sense that  there is a neighborhood $\mathcal{U}$ of $\overline{\Phi}_{(g,\pi)} (g,\pi)$ in $L^p_{-2-q}$ and a constant $C>0$ such that for every $(\phi, V)\in \mathcal{U}$ there is a solution $(\gamma, \tau)$ to 
%\[
%	\overline{\Phi}_{(g,\pi)}(\gamma,\tau) = (\phi, V)
%\]
%and the solution satisfies
%\[
%	\|(\gamma,\tau)-(g,\pi)\|_{W^{2,p}_{-q}\times W^{1,p}_{-1-q}} \le C \| (\phi, V)-\overline{\Phi}_{(g,\pi)}(g,\pi) \|_{L^p_{-2-q}}.
%\] 
\end{lemma}
\begin{remark}
Note the hypothesis that $(g- g_\mathbb{E}, \pi) \in C^{2}_{-q} \times C^{1}_{-1-q}$.
We are grateful for Luen-Fai~Tam and Tin~Yau~Tsang for pointing out an inaccuracy in \cite[Lemma 20]{Eichmair-Huang-Lee-Schoen:2016}: the weaker assumption $(g-g_{\mathbb{E}}, \pi)\in W^{2,p}_{-q} \times W^{1,p}_{-1-q}$ stated in that paper does not seem sufficient to implement the proof given there. Specifically, to apply unique continuation to the adjoint equations in the last paragraph  of the proof of~\cite[p. 111]{Eichmair-Huang-Lee-Schoen:2016} requires an additional hypothesis that $\textup{Ric}_g, \nabla \pi \in C^0_{-2-q}$, as those terms appear in the coefficients of the adjoint equations. The additional regularity hypothesis should also be added in the statement of \cite[Theorem 1]{Eichmair-Huang-Lee-Schoen:2016}.  
%In the current paper, we apply this lemma to an initial data set $(g, \pi)$ at least H\"older regular as assumed in Lemma~\ref{lemma:surjective}
\end{remark}
%\begin{proof}
%We  argue first that $\overline{\Phi}_{(g,\pi)}$ is smooth. For the usual constraint map, the argument in \cite[Lemma 1]{Fischer-Marsden:1975} requires $p>n$, but it has been noted in several places that the assumption can be weakened to $p>n/2$, e.g. \cite{Corvino-Schoen:2006, Chrusciel-Delay:2004, Bartnik:2005}. For the modified constraint operator, one can proceed as in \cite[Proposition 3.1]{Bartnik:2005} to show that $\overline{\Phi}_{(g,\pi)}$ is locally quadratically bounded and has a polynomial structure. Then by standard functional analysis $\overline{\Phi}_{(g,\pi)}$ has continuous Frech\'et derivatives of all orders. 

%The surjectivity of $D\overline{\Phi}_{(g,\pi)}$ has been proven in \cite[Lemma 20]{Eichmair-Huang-Lee-Schoen:2016}, which follows an argument in \cite{Corvino-Schoen:2006} for the usual constraint map in dimension three. Here we provide a simpler proof.

%As shown in \cite[Lemma 20]{Eichmair-Huang-Lee-Schoen:2016}, $D\overline{\Phi}_{(g,\pi)}$ has closed range. It suffices to show that the cokernel of $D\overline{\Phi}_{(g,\pi)}$ is trivial, i.e. the kernel of the formal adjoint operator $(D\overline{\Phi}_{(g,\pi)})^*$ has trivial kernel on $(L^p_{-2-q})^* = L^{p^*}_{}$

%\end{proof}

\section{Dominant energy condition} \label{section:DEC}

The modified constraint operator is designed to preserve the dominant energy condition. In this section, we include a fundamental property of the modified constraint operator (cf. {\cite[Lemma 3.3]{Corvino-Huang:2019}}).

\begin{proposition}\label{proposition:DEC}
Let $(M, g, \pi)$ be an initial data set with $(g- g_\mathbb{E}, \pi) \in C^{2}_{\mathrm{loc}} \times C^{1}_{\mathrm{loc}}$. Assume $(g,\pi)$ satisfies the dominant energy condition $\mu\ge |J|_g$ in $M$. Suppose $(\gamma, \tau)  \in W^{2,p}_{\mathrm{loc}} \times W^{1,p}_{\mathrm{loc}}$ is an initial data set with $|\gamma - g|_g< 3$ in $M$ and  
\[
	\overline{\Phi}_{(g,\pi)}(\gamma, \tau) = \overline{\Phi}_{(g,\pi)}(g, \pi).
\]
Then $(\gamma, \tau)$ also satisfies the dominant energy condition. 

%There is an open ball $U$ centered at $(g,\pi)$ in $W^{2,p}_{-q}\times W^{1,p}_{-1-q}$ such that if $(\gamma, \tau) \in U$ and 
%\[
%\overline{\Phi}_{(g,\pi)}(\gamma, \tau) = \overline{\Phi}_{(g,\pi)}(g, \pi).
%\] 
% Then $(\gamma, \tau)$ satisfies the dominant energy condition. 
% Define  the subset $\mathcal{C}_{(g,\pi)} $ of $\mathscr{M}^{2,p}_{-q}\times W^{1,p}_{-1-q}$  as 
%\[
%	\mathcal{C}_{(g,\pi)} = \left\{ (\gamma, \tau) \in \mathscr{M}^{2,p}_{-q}\times W^{1,p}_{-1-q} : \overline{\Phi}_{(g,\pi)}(\gamma, \tau) = \overline{\Phi}_{(g,\pi)}(g, \pi)\right\}.
%\]
%Then there is an open ball centered at $(g,\pi)$ in  $ \mathcal{C}_{(g,\pi)} $ that consists of initial data sets satisfying the dominant energy condition. 
\end{proposition}
\begin{proof}
Let $(\bar{\mu}, \bar{J})$ be the mass and current densities of $(\gamma, \tau)$. The assumption $ \overline{\Phi}_{(g,\pi)}(\gamma, \tau) = \overline{\Phi}_{(g,\pi)}(g, \pi)$ implies
 \begin{align*}
 	\bar{\mu} &= \mu\\
	\bar{J}^i +\tfrac{1}{2} g^{ij} \gamma_{jk} J^k&=J^i +\tfrac{1}{2} g^{ij} g_{jk} J^k.
\end{align*}
Note that the second identity implies that $\bar{J}$ is at least continuous by using Sobolev embedding for~$\gamma$. Letting $h=\gamma-g$, we have
\[
	\bar{J}^i = J^i - \tfrac{1}{2}  (h\cdot J)^i 
\]
where recall $(h\cdot J)^i = g^{ij} h_{jk} J^k$. We compute, for $|h|_g< 3$, 
\begin{align} \label{equation:DEC}
\begin{split}
	|\bar{J}|_{\gamma}^2 &= \gamma_{ij} \bar{J}^i \bar{J}^j\\
	&=(g_{ij} + h_{ij})\left (J^i - \tfrac{1}{2} (h\cdot J)^i \right) \left (J^j - \tfrac{1}{2} (h\cdot J)^j \right) \\
	&=(g_{ij} + h_{ij}) \left(J^i J^j - g^{il} h_{lk}J^k J^j + \tfrac{1}{4} (h\cdot J)^i (h\cdot J)^j \right)\\
	&= |J|_g^2 -\tfrac{3}{4} |h\cdot J|_g^2 + \tfrac{1}{4} h_{ij} (h\cdot J)^i (h\cdot J)^j\\
	&\le |J|_g^2.
\end{split}
\end{align}
It implies that if $|\gamma - g|_g < 3$ then $(\gamma,\tau)$ satisfies the dominant energy condition $\bar{\mu}\ge |\bar{J}|_{\gamma}$.
\end{proof}

\section{Sobolev version of  positive mass inequality}\label{section:inequality}

For the proof of Theorem~\ref{theorem:main} in the next section, we must show that the positive mass inequality holds with only Sobolev regularity. We will use a density type argument to  approximate an initial data set of Sobolev regularity by a more regular initial data set of type $(p, q, q_0, \alpha)$. As mentioned above in the introduction, the positive mass inequality for asymptotically flat manifolds of type $(p, q, q_0, \alpha)$ was proved in \cite{Eichmair-Huang-Lee-Schoen:2016} for $3\le n \le 7$ and has been announced in \cite{Lohkamp:2016} for $n\ge 3$. 

For the following statement, please refer to Definition~\ref{definition:inequality} in Section~\ref{section:introduction} where we defined what it means for \emph{the positive mass inequality  to hold  near $(g, \pi)$}. 

\begin{theorem}[Sobolev version of  positive mass inequality]\label{theorem:inequality}

Let $(M, g, \pi)$ be asymptotically flat of type $(p, q, q_0, \alpha)$ with the dominant energy condition. Suppose the positive mass inequality holds near $(g, \pi)$.  Then there is an open ball $U$ of $(g, \pi)$ in $W^{2,p}_{-q}\times W^{1,p}_{-1-q}$ such that if $(\gamma, \tau)\in U$ and $\overline{\Phi}_{(g,\pi)}(\gamma,\tau) = \overline{\Phi}_{(g,\pi)}(g, \pi)$, we have
\[
	E(\gamma, \tau) \ge |P(\gamma, \tau)|. 
\] 
\end{theorem}

Before we prove the theorem, we need some preparatory results, and the proof of Theorem~\ref{theorem:inequality} is given after Corollary~\ref{corollary:isomorphism}.

The following lemma is used to solve the modified constraint equations. The proof adapts the argument in \cite[Theorem 1]{Corvino-Schoen:2006}. For a Riemannian metric $g$ and a vector field $Y$, we define
\[
\mathcal{L}_g Y = L_Y g - (\Div_g Y)g.
\]

\begin{lemma}\label{lemma:isomorphism}
Let $(M, g, \pi)$ be an initial data set with $(g-g_{\mathbb{E}}, \pi) \in C^{2}_{-q} \times C^{1}_{-1-q}$. 
Given a function~$u$, a vector field $Y$, a symmetric $(0,2)$-tensor $h$, and a symmetric $(2,0)$-tensor $w$ on $M$, we define
%For $(u, Y) \in W^{2,p}_{-q}$ and $(h, w)\in W^{2,p}_{-q}\times W^{1,p}_{-1-q}$, define the map 
\[
	T(u, Y, h, w) := \overline{\Phi}_{(g,\pi)}((1+u)^{\frac{4}{n-2}} g+h, \pi + \mathcal{L}_g Y +w).
\]
There exists a subspace $W$ of pairs $(u,Y)\in W^{2,p}_{-q}$ and  a finite dimensional subspace  $K\subset W^{2,p}_{-q}\times W^{1,p}_{-1-q}$ of pairs $(h, w)\in C^\infty_c$ such that  
\[
T: W\times K \too L^p_{-2-q}
\]
  is a diffeomorphism from a neighborhood of $0$ in $W \times K$ onto a ball centered at $\overline{\Phi}_{(g,\pi)}(g, \pi)$ in $L^p_{-2-q}$. 
\end{lemma}
\begin{proof}
We define the map $P: W^{2,p}_{-q} \too L^p_{-2-q}$  by
\[
	P(v, Z) = D\overline{\Phi}_{(g, \pi)} (vg, \mathcal{L}_g Z). 
\]
By  \eqref{equation:modified} and \eqref{equation:constraint} (and substituting $(h,w)=(vg, \mathcal{L}_g Z)$ there), the map $P$ (after multiplying an appropriate  constant to the first component of $P$)  is asymptotic to $\Delta_g$  in the sense of \cite[Definition 1.5]{Bartnik:1986} and hence is Fredholm. Let $W$ be a subspace of $W^{2,p}_{-q}$ complementing to the kernel of $P$. 

Because $D\overline{\Phi}_{(g,\pi)}:W^{2,p}_{-q}\times W^{1,p}_{-1-q} \too L^p_{-2-q}$ is surjective by  Lemma~\ref{lemma:surjective}, there is a finite dimensional subspace $K\subset W^{2,p}_{-q}\times W^{1,p}_{-1-q}$, spanned by  linearly independent pairs of tensors $(\eta_1, \xi_1), \cdots, (\eta_N, \xi_N)$,  such that  the image of $K$ by $D\overline{\Phi}_{(g,\pi)}$ complements to the range of $P$, i.e.  $D\overline{\Phi}_{(g,\pi)}(K)\oplus \mathrm{range}(P) = L^p_{-2-q}$.  By smooth approximation, we may assume that all $(\eta_k, \xi_k)\in C^\infty_c$. 

For the map $T$ defined above, we compute its linearization  at $(u, Y, h, w) =0$:
\[
	DT|_{0}(v, Z, \eta, \xi) = D\overline{\Phi}_{(g, \pi)} (\tfrac{4}{n-2} v g, \mathcal{L}_g Z) + D\overline{\Phi}_{(g, \pi)} (\eta,\xi). 
\]
The linearization is an isomorphism by construction. The desired statement follows from inverse function theorem. 
\end{proof}

The following corollary is a direct consequence of the fact that  a linear operator that is sufficiently close (in the operator norm) to an isomorphism is also an isomorphism.

\begin{corollary}\label{corollary:isomorphism}
Let $(M, g, \pi)$ be an initial data set with $(g-g_{\mathbb{E}}, \pi) \in C^{2}_{-q} \times C^{1}_{-1-q}$ and  $W, K$  the corresponding function spaces  defined as in Lemma~\ref{lemma:isomorphism}. For an initial data set $(\gamma,\tau)\in \mathscr{M}^{2,p}_{-q}\times W^{1,p}_{-1-q}$, we define the map  $T_{(\gamma,\tau)}: W \times K \too L^p_{-2-q}$ by
\[
	T_{(\gamma,\tau)}(u, Y, h, w) := \overline{\Phi}_{(g,\pi)}((1+u)^{\frac{4}{n-2}} \gamma+h, \tau + \mathcal{L}_\gamma Y +w).
\]
Then there is $ \delta>0, C_1>0$ and an open ball $U$ centered at $(g, \pi)$ in $W^{2,p}_{-q}\times W^{1,p}_{-1-q}$ such that for each $(\gamma, \tau)\in U$, the map $T_{(\gamma,\tau)}$ is a diffeomorphism from a neighborhood  $B$ of $0$ in $W\times K$ onto the open ball centered at $\overline{\Phi}_{(g,\pi)}(g, \pi)$ of radius $\delta$ in $L^p_{-2-q}$, and, for all $(u, Y, h, w)\in B$, 
\[
	\|(u, Y, h, w)\|_{W\times K} \le C_1 \| T_{(\gamma,\tau)}(u, Y, h, w) - T_{(\gamma,\tau)}(0) \|_{L^p_{-2-q}}.
\]

\end{corollary}
\begin{proof}
Observe that our notation says that $T_{(g,\pi)}= T$ where $T$ is the map defined in Lemma~\ref{lemma:isomorphism}.  For $U$ sufficiently small, the linearization of $T_{(\gamma, \tau)} $ at $0$  is close (in the operator norm) to the linearization of $T$ at $0$, and hence $\left.DT_{(\gamma, \tau)}\right|_{0}$ is also an isomorphism. 

By inverse function theorem, $T_{(\gamma,\tau)}$ is a diffeomorphism from a neighborhood of $0$ onto a ball centered at  $T_{(\gamma, \tau)}(0) = \overline{\Phi}_{(g,\pi)}(\gamma,\tau)$ of radius $2\delta$, which contains the ball centered at $\overline{\Phi}_{(g,\pi)}(g, \pi)$ of radius $\delta$,  for $(\gamma,\tau)$ sufficiently close to $(g, \pi)$.  Note that $\delta$ depends only on the upper bounds of $\|DT_{(\gamma, \tau)}\|$ and $\|D^2 T_{(\gamma, \tau)}\|$. Since there is a uniform bound on $\|DT_{(\gamma, \tau)}\|$ and $\|D^2 T_{(\gamma, \tau)}\|$ for $(\gamma, \tau)$ in a sufficiently small neighborhood $U$, this radius $\delta$ can be chosen to be uniform in $(\gamma,\tau)$. The desired estimate follows from the fact that the inverse map $T^{-1}_{(\gamma,\tau)}$ is  locally Lipschitz with a uniform Lipschitz constant for $(\gamma, \tau)\in U$. 
\end{proof}

We now prove the main result of this section. 

\begin{proof}[Proof of Theorem~\ref{theorem:inequality}]
We first outline the proof. We will approximate $(\gamma, \tau)$ by initial data sets $(\bar{\gamma}_k, \bar{\tau}_k)$ of H\"older regularity and with the dominant energy condition. By hypothesis, positivity of the ADM energy-momentum  for $(\bar{\gamma}_k, \bar{\tau}_k)$ holds. Then the desired ADM energy-momentum positivity for $(\gamma, \tau)$ follows from continuity of the ADM energy-momentum. 

The main point is to construct $(\bar{\gamma}_k, \bar{\tau}_k)$ that satisfies the dominant energy condition. Let $U\subset W^{2,p}_{-q}\times W^{1,p}_{-1-q}$ be the ball centered at $(g, \pi)$ from Corollary~\ref{corollary:isomorphism}, and let $(\gamma, \tau)\in U$. By smooth approximation, there is a sequence of $C^\infty_{\mathrm{loc}}$ initial data sets $(\gamma_k, \tau_k) \in U$  and $(\gamma_k, \tau_k) \to (\gamma, \tau)$ in $W^{2,p}_{-q}\times W^{1,p}_{-1-q}$. 

Applying Corollary~\ref{corollary:isomorphism} for $(\gamma_k, \tau_k)$, we find $(u_k, Y_k, h_k, w_k)\in W\times K$ such  that 
\begin{align}\label{equation:smooth-approximation}
	\overline{\Phi}_{(g,\pi)}((1+u_k)^{\frac{4}{n-2}} \gamma_k + h_k, \tau_k + \mathcal{L}_{\gamma_k} Y_k + w_k) =\overline{ \Phi}_{(g,\pi)} (g, \pi)
\end{align}
and
\[
	\| (u_k, Y_k, h_k, w_k )\|_{W\times K} \le C_1\|\overline{ \Phi}(g, \pi)-\overline{ \Phi}(\gamma_k, \tau_k) \|_{L^p_{-2-q}}.
\]
The  assumption $\overline{ \Phi}(\gamma, \tau)=\overline{ \Phi}(g, \pi)$ implies that 
\[
	\| (u_k, Y_k, h_k, w_k )\|_{W\times K}\to 0 \mbox{ as } k\to \infty. 
\]
Denote 
\[
	\bar{\gamma}_k = (1+u_k)^{\frac{4}{n-2}} \gamma_k + h_k\quad \mbox{ and } \quad \bar{\tau}_k= \tau_k + \mathcal{L}_{\gamma_k} Y_k + w_k.
\]
We have shown that $\overline{\Phi}_{(g,\pi)}(\bar{\gamma}_k, \bar{\tau}_k) =\overline{ \Phi}_{(g,\pi)}(g, \pi)$ and $(\bar{\gamma}_k, \bar{\tau}_k)\to (\gamma, \tau)$ in $W^{2,p}_{-q}\times W^{1,p}_{-1-q}$. By Proposition~\ref{proposition:DEC}, we see that $(\bar{\gamma}_k, \bar{\tau}_k)$ satisfies the dominant energy condition for $k$ sufficiently large. (We may further shrink $U$ to ensure $|\gamma-g|_g < 3$.)

By shrinking $\alpha$ if necessary, we  assume $\alpha \in (0, 1-\frac{n}{p}]$. We will show that $(\bar{\gamma}_k-g_{\mathbb{E}}, \bar{\tau}_k)\in C^{2,\alpha}_{-q}\times C^{1,\alpha}_{-1-q}$. Since $(\gamma_k, \tau_k)$ and $(h_k, w_k)$ are smooth with the appropriate fall-off rates, it suffices to show that $(u_k, Y_k) \in C^{2,\alpha}_{-q}$. Equation \eqref{equation:smooth-approximation} is a quasi-linear elliptic PDE system of $(u_k, Y_k)$ where the terms with top order derivatives are $\Delta_{\gamma_k } u_k$ and $\Delta_{\gamma_k} Y_k$ (after multiplying a factor of an appropriate power of $(1+u_k)$ to the equations). By Sobolev embedding, $(u_k, Y_k)\in C^{1,\alpha}_{-q}$. Then it is direct to see the terms with lower order derivatives in the PDE system are in $C^{0,\alpha}_{-2-q}$. That is, $(u_k, Y_k)$  satisfies $(n+1)$ Poisson equations $\Delta_{\gamma_k} (u_k, Y_k) \in C^{0,\alpha}_{-2-q}$. Standard elliptic regularity implies the desired H\"older regularity for $(u_k, Y_k)$. 

Thus $(\bar{\gamma}_k, \bar{\tau}_k)$ is regular enough to guarantee that $E(\bar{\gamma}_k, \bar{\tau}_k) \ge |P(\bar{\gamma}_k, \bar{\tau}_k)|$. Since $(\bar{\gamma}_k, \bar{\tau}_k)$ converges to $(\gamma,\tau)$ in $W^{2,p}_{-q}\times W^{1,p}_{-1-q}$ with $\overline{\Phi}_{(g,\pi)}(\bar{\gamma}_k, \bar{\tau}_k) = \overline{\Phi}_{(g,\pi)}(\gamma, \tau)$, using the continuity of the ADM energy-momentum (see, e.g. \cite[Proposition 19]{Eichmair-Huang-Lee-Schoen:2016}), we conclude that $E(\gamma, \tau) \ge |P(\gamma, \tau)|$. 

\end{proof}

In the rest of this section, we include a deformation result of independent interest. This result,   first proven in \cite[Theorem 22]{Eichmair-Huang-Lee-Schoen:2016} by linear approximation, was an important step to obtain harmonic asymptotics in that paper. Here we provide an alternative proof by using the modified constraint operator. The deformation result has general applications, although note that it is not used elsewhere in the current paper.

\begin{theorem}\label{theorem:strict}
Let $(M, g, \pi)$ be asymptotically flat of type $(p, q, q_0, \alpha)$ with  mass and current densities $(\mu, J)$. There exists $\lambda_0>0, C_1>0$ such that for each $0< \lambda < \lambda_0$, there exists an initial data set $(\bar{g}, \bar{\pi})$ of the same type with $\|(\bar{g}, \bar{\pi} )- (g, \pi) \|_{W^{2,p}_{-q} \times W^{1,p}_{-1-q}} < C_1 \lambda$ such that 
\[
	\bar{\mu}  > (1+\lambda) |\bar{J}|_{\bar{g}} + (1+\lambda) (\mu - |J|_g)
\]
where $(\bar{\mu}, \bar{J})$ are the mass and current densities of $(\bar{g}, \bar{\pi})$. As a consequence, if $\mu\ge |J|_g$, then $(\bar{g}, \bar{\pi})$ satisfies the strict dominant energy condition with 
\[
	\bar{\mu} -  |\bar{J}|_{\bar{g}} > \lambda |\bar{J}|_{\bar{g}}.
\]
\end{theorem}
\begin{proof}
By Lemma~\ref{lemma:isomorphism}, the map $T: W\times K\too L^p_{-2-q}$ defined by 
\[
	T(u, Y, h, w) := \overline{\Phi}_{(g,\pi)}((1+u)^{\frac{4}{n-2}} g+h, \pi + \mathcal{L}_g Y +w)
\]
is a diffeomorphism from an open neighborhood of $(u, Y, h, w)=0$ to an open ball centered at $\overline{\Phi}(g, \pi)$ of some radius $\delta>0$. 

 Given a smooth function $\phi >0$ with $|\phi(x) |\le |x|^{-n-q_0}$ outside a compact subset of $M$, there exists a positive number $\lambda_0$ such that
\[
 \lambda_0 \left(\|  \phi \|_{L^p_{-2-q}} + \| \mu\|_{L^p_{-2-q} } \right) < \delta. 
\]
Since $T$ is a local diffeomorphism, there is a constant $C_1>0$ such that for each $0< \lambda < \lambda_0$, there is $(u, Y, h, w)$ that satisfies
\begin{align} \label{equation:approximation}
	\overline{\Phi}_{(g, \pi)}((1+u)^{\frac{4}{n-2}} g+h, \pi + \mathcal{L}_g Y +w)=\overline{\Phi}_{(g,\pi)} (g, \pi) +( \lambda (\mu + \phi), 0)
\end{align}
with 
\[
	\| (u, Y, h, w)\|_{W\times K} \le C_1  \| \lambda (\mu + \phi)\|_{L^p_{-2-q} } \le C_1 \lambda.
\]

We define $\bar{g} = (1+u)^{\frac{4}{n-2}} g+h$ and $\bar{\pi}=\pi + \mathcal{L}_g Y +w$. By applying elliptic regularity to the quasi-linear equations \eqref{equation:approximation} of $(u, Y)$  (just as in the proof of Theorem~\ref{theorem:inequality}), we have $(u, Y)\in C^{2,\alpha}_{-q}$ and thus one can directly verify that $(\bar{g}, \bar{\pi})$ is of type $(p, q, q_0, \alpha)$. It remains to show the desired inequality. Equation \eqref{equation:approximation} implies 
\begin{align*}
	\bar{\mu} &= (1+\lambda) \mu + \lambda \phi  \\
	\bar{J}^i +\tfrac{1}{2} g^{ij} \gamma_{jk} J^k&=J^i +\tfrac{1}{2} g^{ij} g_{jk} J^k.
\end{align*}
Compute as in \eqref{equation:DEC}, we obtain
\[
	|\bar{J}|_{\bar{g}} \le |J|_g
\]
provided $\lambda_0$ sufficiently small so that $|\bar{g} - g|< 3$. We now conclude
\[
	\bar{\mu}  - (1+\lambda) |\bar{J}|_{\bar{g}} > (1+\lambda) (\mu - |J|_g ).
\]
\end{proof}

\section{Main argument}\label{section:variational}

%Before we define the functional considered in the paper, we make a general remark about the extension of constant vectors defined at infinity. For $a\in \mathbb{R}$ and $b\in \mathbb{R}^n$, we can define a smooth function $f_0$ and a smooth vector field $X_0$ so that $f_0=a$ and $X_0=b$ in $M\setminus K$ with respect to the chart at infinity. Multiplying $(f_0, X_0)$ by a smooth cut-off function that is identically zero in $K$, we obtain a pair that is defined globally in $M$.

We introduce a modification of the classical Hamiltonian defined by Regge and Teitelboim~\cite{Regge-Teitelboim:1974} (see also  \cite[Section 5]{Bartnik:2005}) by employing the modified constraint operator in place of the usual constraint operator.
 \begin{definition}\label{definition:functional}
 Let $(M, g, \pi)$ be asymptotically flat of type $(p, q, q_0, \alpha)$.  Let $a\in \mathbb{R}$ and $b\in \mathbb{R}^n$. Let $(f_0, X_0)$ be a pair of a function and a vector field on $M$ (which we will often call a \emph{lapse-shift pair}) such that $(f_0, X_0)$ is smooth and is equal to $(a,b)$ in the exterior coordinate chart for $M\setminus K$. 

We define the \emph{modified Regge-Teitelboim Hamiltonian} $\mathcal{H} :\mathscr{M}^{2,p}_{-q}\times W^{1,p}_{-1-q} \too \mathbb{R} $ corresponding to $(g,\pi)$ and $(f_0, X_0)$ by 
\begin{align}\label{equation:functional}
	\mathcal{H} (\gamma, \tau) =(n-1)\omega_{n-1} \left[2 a E(\gamma,\tau) + b \cdot P(\gamma, \tau)\right] - \int_M \overline{\Phi}_{(g,\pi)}(\gamma, \tau) \cdot (f_0, X_0)\, d\mu_g
\end{align}
where the volume measure $d\mu_g$ and the inner product in the integral are both with respect to~$g$.
\end{definition}

Although two terms in the expression given above are not individually well-defined for arbitrary $(\gamma, \tau) \in \mathscr{M}^{2,p}_{-q}\times W^{1,p}_{-1-q}$ (because the corresponding integrals may not converge), it is well-known that the functional $\mathcal{H}$ described above extends to all of $\mathscr{M}^{2,p}_{-q}\times W^{1,p}_{-1-q}$ in a natural way. 
We simply use the following alternative expression by rewriting the  ADM energy-momentum surface integrals as volume integrals via divergence theorem and rearranging terms:
\begin{align}\label{equation:H}
\begin{split}
	\mathcal{H} (\gamma, \tau) &=\int_M \left[\left(\Div_g [\Div_{g_{\mathbb{E}}} \gamma - d(\mathrm{tr}_{g_{\mathbb{E}}} \gamma) ], \Div_g \tau\right)- \Phi(\gamma, \tau)- \left(0, \tfrac{1}{2} \gamma \cdot J\right) \right] \cdot (f_0, X_0)\, d\mu_g\\
	&\quad + \int_M \left( [\Div_{g_\mathbb{E}} \gamma - d(\mathrm{tr}_{g_{\mathbb{E}}} \gamma) ], \tau \right)\cdot (\nabla f_0, \nabla X_0) \, d\mu_g
	\end{split}
\end{align}
where recall that $g_\mathbb{E}$ is a background metric equal to the Euclidean one on the exterior coordinate chart.
The second integral is finite because $|\nabla f_0|, |\nabla X_0|= O(|x|^{-1-q})$. Asymptotic flatness of $(g,\pi)$ implies that $J=\Div_{g} \pi$ is integrable. Meanwhile the integrability of $\left(\Div_g [\Div_{g_{\mathbb{E}}} \gamma - d(\mathrm{tr}_{g_{\mathbb{E}}} \gamma) ], \Div_g \tau\right)- \Phi(\gamma, \tau)$ is a standard fact, which can be verified by writing out the expression in the exterior coordinate chart and using the assumed decay rates. The point is that the first term matches the top-order part of 
$\Phi(\gamma, \tau)$ and the other terms decay fast enough to ensure integrability.

%Observe that in the definition of $\mathcal{H}$, the pair 
% for the functional to also defined (of difference functional values) if one replaces $(f_0, X_0)$ with any smooth pair $(f_0', X_0')$ satisfying $(f_0'-f_0, X_0'-X_0) \in L^{p^*}_{-n+2+q}$ where $p* = \frac{p}{p-1}$. 

We compute the first variation of the functional (Cf. {\cite[Theorem 5.2]{Bartnik:2005}}). 

\begin{lemma}\label{lemma:Hamiltonian}
 Let $(M, g, \pi)$ be asymptotically flat initial data set of type $(p, q, q_0, \alpha)$. Let $a\in \mathbb{R}$ and $b\in \mathbb{R}^n$, and let $(f_0, X_0)$ be a smooth lapse-shift pair such that $(f_0, X_0) = (a,b)$ on the exterior coordinate chart for $M\setminus K$. 

Let $\mathcal{H}:\mathscr{M}^{2,p}_{-q}\times W^{1,p}_{-1-q} \too \mathbb{R} $ be the modified Regge-Teitelboim Hamiltonian corresponding to $(g,\pi)$ and $(f_0, X_0)$. Then $\mathcal{H}$ is differentiable at $(g,\pi)$ with derivative given by
\[
	 D\mathcal{H}\big|_{(g,\pi)} (h, w) =- \int_M (h, w) \cdot (D\overline{\Phi}_{(g,\pi)})^*(f_0, X_0) \, d\mu_g
\]
for all $(h,w) \in W^{2,p}_{-q}\times W^{1,p}_{-1-q} $. 
\end{lemma}
\begin{proof}
The argument is essentially the same as in \cite[Theorem 5.2]{Bartnik:2005} for the usual Regge-Teitelboim Hamiltonian, but we summarize the computation here for the sake of completeness. Differentiability of $\mathcal{H}$ comes from local boundedness of $\mathcal{H}$ and the polynomial structure of the integrand.  To derive the linearization, we linearize \eqref{equation:H} and have, for all $(h,w) \in W^{2,p}_{-q}\times W^{1,p}_{-1-q}$,
\begin{align}\label{equation:linearization-H}
\begin{split}
	D\mathcal{H}|_{(g,\pi)} (h, w)  &=\int_M \left[\left(\Div_g [\Div_{g_{\mathbb{E}}} h - d(\mathrm{tr}_{g_{\mathbb{E}}} h) ], \Div_g  w\right)- D\overline{\Phi}_{(g,\pi)}(h, w) \right] \cdot (f_0, X_0)\, d\mu_g\\
	&\quad + \int_M \left( [\Div_{g_\mathbb{E}} h - d(\mathrm{tr}_{g_{\mathbb{E}}} h) ], w\right)\cdot (\nabla f_0, \nabla X_0) \, d\mu_g.
	\end{split}
\end{align}

%\begin{aligned} \mathcal {H} (\gamma , \tau )= & {} \int _M \left[ \left( {{\,\mathrm{div}\,}}_g [{{\,\mathrm{div}\,}}_{g_{\mathbb {E}}} \gamma - d({\mathrm {tr}}_{g_{\mathbb {E}}} \gamma ) ], {{\,\mathrm{div}\,}}_g \tau \right) - \Phi (\gamma , \tau )\right]  \cdot (f_0, X_0)\, d\mu _g  \nonumber \\
%&- \int_M \left( 0, \tfrac{1}{2} \gamma \cdot J\right) \cdot (f_0, X_0)\, d\mu _g  \nonumber\\
%&+ \int _M \left( [{{\,\mathrm{div}\,}}_{g_\mathbb {E}} \gamma - d({\mathrm {tr}}_{g_{\mathbb {E}}} \gamma ) ], \tau \right) \cdot (\nabla f_0, \nabla X_0) \, d\mu _g \end{aligned}

%\begin{align}\label{equation:linearization-H}
%\begin{split}
%	D\mathcal{H}|_{(g,\pi)} (h, w) &= \int_M \left[\left(\nabla_j (h_{ij,i} - h_{ii,j}), \Div_g w\right)- D\overline{\Phi}_{(g,\pi)}(h, w)\right] \cdot (f_0, X_0)\, d\mu_g\\
%	&\quad + \int_M \left[(\nabla_j f_0)(h_{ij,i} -h_{ii,j}) +( \nabla_j X_0^i) w_{ij, j}\right]\, d\mu_g.
%	\end{split}
%\end{align}
By the definition of the $L^2$ adjoint operator and the divergence theorem, we obtain
\begin{align*}
D\mathcal{H}|_{(g,\pi)} (h, w)  &= 
\lim_{r\to\infty}\left\{-\int_{|x|<r} (h, w) \cdot (D\overline{\Phi}_{(g,\pi)})^*(f_0, X_0)\, d\mu_g \right.\\
&\qquad \qquad \left.+\int_{|x|=r} \left[ (\Div_{g_{\mathbb{E}}} h - d(\mathrm{tr}_{g_{\mathbb{E}}} h), w)\cdot (f_0, X_0) - B\right]_i \nu^i\, d\mathcal{H}^{n-1}\right\}
\end{align*}
where $B$ is the boundary integrand that arises from taking the adjoint of $D\overline{\Phi}_{(g,\pi)}$. The upshot is that $B$ equals $(\Div_{g_{\mathbb{E}}} h - d(\mathrm{tr}_{g_{\mathbb{E}}} h), w)\cdot (f_0, X_0)$ modulo terms that decay fast enough so that the boundary integral above vanishes as $r\to\infty$.

%Technically, this argument only works when the ADM energy-momentum is well-defined. In general, it works because we have broken the expression down into pieces which are individually well-defined for any $(\gamma, \tau)\in \mathscr{M}^{2,p}_{-q}\times W^{1,p}_{-1-q} $. Integrating by parts in the first integral of \eqref{equation:linearization-H}, we will obtain an integral over $M$ that differs from $(D\overline{\Phi}_{(g,\pi)})^*(f_0, X_0)$ by some top-order expression, but the boundary integral will be zero. But then integrating by parts in the  second integral of \eqref{equation:linearization-H} will cancel against that exact same top-order expression.\textcolor{red}{I couldn't think of a nice way to explain this, so I kind of hand-waved my way through it, but does it make sense? Is it essentially correct?}

\end{proof}

Now, we assume that $(g, \pi)$  satisfies the dominant energy condition and $E=|P|$. We would like to show that $(g,\pi)$ locally minimizes its corresponding modified Regge-Teitelboim Hamiltonian over its $\overline{\Phi}_{(g,\pi)}$ level set, which gives rise to an asymptotically translational lapse-shift pair lying in the kernel of $(D\overline{\Phi}_{(g,\pi)})^*$.

\begin{theorem}\label{theorem:nontrivial-kernel}
Let $(M, g, \pi)$ be asymptotically flat of type $(p, q, q_0, \alpha)$ satisfying the dominant energy condition. Assume that the  positive mass inequality holds near $(g, \pi)$.  If $E = |P|$, then there exists a lapse-shift pair $(f, X)\in C^{2,\alpha}_{\mathrm{loc}}(M)$ solving 
\begin{align*}
&(D\overline{\Phi}_{(g,\pi)})^*(f, X)=0 \quad \mbox{ in } M\\
&(f, X)= (E, -2P) + O^{2,\alpha}(|x|^{-q}). 
\end{align*}
\end{theorem}
\begin{proof}  
Let $(f_0, X_0)$ be a smooth lapse-shift pair such that $(f_0, X_0) = (E, -2P)$ on the exterior coordinate chart for $M\setminus K$, where $(E, P)$ denotes the ADM energy-momentum of $(g,\pi)$. Let $\mathcal{H}:\mathscr{M}^{2,p}_{-q}\times W^{1,p}_{-1-q} \too \mathbb{R}$ be the modified Regge-Teitelboim Hamiltonian corresponding to $(g,\pi)$ and $(f_0, X_0)$. 

Define
\[
	\mathcal{C}_{(g,\pi)} = \left\{ (\gamma, \tau) \in \mathscr{M}^{2,p}_{-q}\times W^{1,p}_{-1-q} : \overline{\Phi}_{(g,\pi)}(\gamma, \tau) = \overline{\Phi}_{(g,\pi)}(g, \pi)\right\}.
\]
We claim that that $(g,\pi)$ is a local minimizer of $\mathcal{H}$  in $\mathcal{C}_{(g,\pi)}$. Recall the definition of $\mathcal{H}$ in \eqref{equation:functional} with $(a, b) = (E, -2P)$. Note that  $\overline{\Phi}_{(g,\pi)}(\gamma, \tau)$ is integrable for $(\gamma, \tau) \in \mathcal{C}_{(g, \pi)}$, and thus the two terms in the functional $\mathcal{H}$ are individually well-defined.  It is clear that the integral term in the functional has the same value for all $(\gamma, \tau) \in \mathcal{C}_{(g,\pi)}$. It suffices to show that the local minimum of the ADM energy-momentum term is zero and is realized by $(g,\pi)$. By Proposition~\ref{proposition:DEC}, the Sobolev version of the positive mass inequality (Theorem~\ref{theorem:inequality}) applies to show that
 \[E(\gamma, \tau) \ge |P(\gamma,\tau)|\] 
 for any $(\gamma, \tau)$ in a neighborhood of $(g,\pi)$ in $\mathcal{C}_{(g,\pi)}$.  We compute
\[
	E E(\gamma,\tau) - P \cdot P(\gamma, \tau) \ge E  E(\gamma,\tau) - |P||P(\gamma,\tau)| = E( E(\gamma,\tau) - |P(\gamma, \tau)|) \ge 0
\]
with equality at $(g,\pi)$, thus establishing our claim. 

In other words, $(g,\pi)$ locally minimizes $\mathcal{H}$ constrained to a level set of $\overline{\Phi}_{(g,\pi)}$. Using surjectivity of $D\overline{\Phi}_{(g,\pi)}$ (Lemma~\ref{lemma:surjective}), we can now apply the method of Lagrange multipliers (Theorem~\ref{theorem:Lagrange}) to see that there exists $(f_1, X_1)\in (L^p_{-2-q})^* = L^{p^*}_{-n+2+q}$ where $p^* = p/(p-1)$, such that for all $(h,w)\in W^{2,p}_{-q}\times W^{1,p}_{-1-q} $,
\[
		\left. D\mathcal{H} \right|_{(g, \pi)} (h, w)  =\int_M (f_1, X_1) \cdot D\overline{\Phi}_{(g,\pi)}(h, w) \, d\mu_g.
\]
Replacing the left hand side with the formula of the derivative of $\mathcal{H}$ in Lemma~\ref{lemma:Hamiltonian}, we obtain, for all $(h,w)\in W^{2,p}_{-q}\times W^{1,p}_{-1-q} $,
\[
	-\int_M (h,w) \cdot (D\overline{\Phi}_{(g,\pi)})^* (f_0, X_0)  \, d\mu_g= \int_M (f_1, X_1) \cdot D\overline{\Phi}_{(g,\pi)}(h, w) \, d\mu_g.
\]
In particular, the above identity holds for  $(h,w)\in C^\infty_c$. It means that $(f_1, X_1)\in  L^{p^*}_{-n+2+q}$ weakly solves 
\[
-(D\overline{\Phi}_{(g,\pi)})^*(f_0, X_0) = (D\overline{\Phi}_{(g,\pi)})^*(f_1, X_1).
\]
Our asymptotic assumptions imply that $(D\overline{\Phi}_{(g,\pi)})^*(f_0, X_0) \in  C^{0,\alpha}_{-2-q}\times C^{1,\alpha}_{-1-q}$, and therefore we can invoke elliptic regularity (Proposition~\ref{proposition:regularity}) to conclude  that $(f_1, X_1)\in C^{2,\alpha}_{-q}$. Setting $(f, X) =(f_0, X_0) + (f_1, X_1)$ gives us the desired statement.  
\end{proof}

%According to Beig and Chru\'{s}ciel \cite{Beig-Chrusciel:1997}, Killing initial data on an initial data set $(M, g,\pi)$ is a lapse-shift pair $(f, X)$ that has the property that it gives rise to a meaningful \emph{spacetime Killing development} (as defined in \cite{Beig-Chrusciel:1996}) wherever $f$ is non-vanishing. That is, $(g,\pi)$ embeds into a spacetime in such a way that $(f, X)$ gives rise to a spacetime Killing vector. (Unfortunately, because of different conventions, our equations for $(f, X)$ differ from the ones for the more physical lapse-shift $(N, Y)$ used in \cite{Beig-Chrusciel:1996, Beig-Chrusciel:1997} by various factors of $2$ and $-1$.) This property translates into the condition
%\[ -\tfrac{1}{2} (L_Xg)^{ij} + \left(\tfrac{2}{n-1} (\mbox{tr}_g \pi ) g^{ij}- 2 \pi^{ij}  \right) f=0,\]
%and hence they define \emph{Killing initial data} to be any solution $(f,X)$ of this set of equations. Note that these are the exact same equations that appear as the \emph{second} set of equations of $D\Phi|_{(g,\pi)}^*(f, X)=(0,0)$, and consequently, the same ones as in $(D\overline{\Phi}_{(g,\pi)})^*(f, X)=(0,0)$. We should also mention an observation originally due to V.~Moncrief \cite{Moncrief:1975} that, for \emph{vacuum} initial data sets, if \emph{all} of the equations $D\Phi|_{(g,\pi)}^*(f, X)=(0,0)$ hold, then the Killing development matches the spacetime evolution according to the vacuum Einstein equations. Thus a lapse-shift satisfying $D\Phi|_{(g,\pi)}^*(f, X)=(0,0)$ could be called vacuum Killing initial data.

To complete the proof of Theorem~\ref{theorem:main}, we use the following theorem, which is a corollary of Beig and Chru\'sciel's Theorem 3.4 in \cite{Beig-Chrusciel:1996}. 
\begin{theorem}\label{theorem:kernel}
Let $(M, g, \pi)$ be  asymptotically flat of type $(p, q, q_0, \alpha)$. Suppose the ADM energy-momentum $E=|P|$.  Let $(f,X)$ solve $(D\overline{\Phi}_{(g,\pi)})^*(f, X)=0$ in  $M$ with $(f, X)= (E, -2P) +O^{2,\alpha}(|x|^{-q})$. Then $E=|P|=0$. 
\end{theorem}

We provide a complete proof of Theorem~\ref{theorem:kernel} in Appendix~\ref{section:KID}. Our proof adapts the original argument of  \cite{Beig-Chrusciel:1996} except that we derive general expansions for $(f,X)$, which may have other applications.  (See Theorem~\ref{theorem:asymptotics} and Corollary~\ref{corollary:asymptotics}.) 

\begin{proof}[Proof of Theorem~\ref{theorem:main}]

The variational argument in Theorem~\ref{theorem:nontrivial-kernel} produces the lapse-shift pair $(f, X)$ that satisfies the hypotheses of Theorem~\ref{theorem:kernel}. We then conclude $E=|P|=0$. 
\end{proof}

\appendix

\section{Asymptotically Killing lapse-shift pair}\label{section:KID}

In this section, we prove Theorem~\ref{theorem:kernel}, originally due to  Beig and Chru\'sciel in \cite[Section III]{Beig-Chrusciel:1996}. We extend the argument to a slightly more general statement in Theorem~\ref{theorem:zero} below. 

\begin{definition}\label{defin:Killing}
Let $(M, g, \pi)$ be asymptotically flat of type $(p, q, q_0, \alpha)$, not necessarily vacuum. We say that a lapse-shift pair $(f,X)$ defined on the exterior region $M\setminus K$ is \emph{asymptotically vacuum Killing initial data} for $(g,\pi)$ if
\[ 
D\Phi|_{(g,\pi)}^*(f, X) \in C^{0,\alpha}_{-n-q_0} \times C^{1,\alpha}_{-1-2q}.
\]
Furthermore we say that $(f,X)$ is \emph{asymptotically translational} if there exists $a\in \mathbb{R}$ and $b\in \mathbb{R}^n$ such that 
\[
(f, X) =  (a,b)+O^{2,\alpha}(|x|^{-q}).
\]
In this case we say that $(f,X)$ \emph{is asymptotic to} $(a,b)$.
\end{definition}

The main goal of this section is to prove the following theorem, of which Theorem~\ref{theorem:kernel} is an easy corollary.
\begin{thm}\label{theorem:zero}
Let $(M, g, \pi)$ be  asymptotically flat initial of type $(p, q, q_0, \alpha)$. Suppose the ADM energy-momentum $E=|P|$.  Let $(f,X)$ be asymptotically vacuum Killing initial data for $(g,\pi)$ that is asymptotic to some $(a,b)$, where $a\in \mathbb{R}$, $b\in \mathbb{R}^n$.  If $a$ and $b$ are not both zero, then $E=|P|=0$.
\end{thm}

The proof is presented at the end of this section, after Corollary~\ref{corollary:asymptotics}.

All of our computations will take place in the exterior coordinate chart  $M\setminus K \cong \mathbb{R}^n \setminus B$, where $B$ is a closed unit ball centered at the origin in $\mathbb{R}^n$.  For notation, a comma in the subscript means ordinary differentiation in the coordinate chart (which is the same as covariant differentiation with respect to $g_\mathbb{E}$),
 and $\Delta_0, \mathrm{tr}_0, \Div_0$ are, respectively, the usual Euclidean Laplacian, trace, and divergence operators. We sum over repeated indices, unless otherwise indicated. 
 %Since these computations are taking place in Euclidean space, we will use lowered indices as our default, though repeated index summation is still in effect.
 We also write $\int_{S_\infty} u\, d\mathcal{H}^{n-1}$ as shorthand for $\lim_{r\to\infty} \int_{|x|=r} u\, d\mathcal{H}^{n-1}$.
We start with some computational lemmas.

\begin{lemma}\label{lemma:basic}
Let $T_{ij}\in C^1_{\mathrm{loc}}(\mathbb{R}^n\setminus B)$ be a $2$-tensor. Then
\begin{align*}
 \int_{|x|=r}   T_{ij,j} \nu_i\, d\mathcal{H}^{n-1} = \int_{|x|=r}   T_{ji,j} \nu_i\, d\mathcal{H}^{n-1} .
\end{align*}
\end{lemma}
\begin{proof}
The key observation is that $(T_{ij}-T_{ji})\nu_i$ is tangential to $|x|=r$, and thus the divergence theorem on the sphere tells us that
\begin{align*}
0 &=  \int_{|x|=r}\left[(T_{ij} - T_{ji})\nu_i\right]_{,j}\, d\mathcal{H}^{n-1} \\ 
&=\int_{|x|=r}(T_{ij,j} - T_{ji,j})\nu_i\, d\mathcal{H}^{n-1} + \int_{|x|=r} (T_{ij}-T_{ji}) \nu_{i,j}\, d\mathcal{H}^{n-1}\\
&= \int_{|x|=r}(T_{ij,j} - T_{ji,j})\nu_i\, d\mathcal{H}^{n-1},
\end{align*}
where the last equality follows from symmetry considerations. 
\end{proof}

\begin{corollary}\label{corollary:basic}
For any function $f\in C^2_{\mathrm{loc}}(\mathbb{R}^n\setminus B)$,
\begin{align*}
	\int_{|x|=r} (\Delta_0 f ) \nu_j\, d\mathcal{H}^{n-1}&=\int_{|x|=r} f_{,ij}  \nu_i\, d\mathcal{H}^{n-1}\quad \mbox{ for }j = 1, 2, \dots, n\\
	\int_{|x|=r} (x\cdot \nu) \Delta_0 f \, d\mathcal{H}^{n-1}&=\int_{|x|=r} (f_{,ij} x_j 
 + (n-1)f_{,i})
	\nu_i \, d\mathcal{H}^{n-1}.
\end{align*}
\end{corollary}
\begin{proof}
Fixing $j$ and applying the previous lemma to $T_{ik} = f_{,k}\delta_{ij} $ gives the first equality. For the second equality, we set $T_{ij} = f_{,j}x_i$ and apply the previous lemma. 

\end{proof}
%\begin{proof}
%By divergence theorem and noting the vector $(T_{ij} - T_{ji})\nu^i$ is tangential to $|x|=r$ by anti-symmetry in the indices $i,j$, we have 
%\begin{align*}
%	\int_{|x|=r}(T_{ij} - T_{ji})_{,j}\nu^i\, d\mathcal{H}^{n-1}&= \int_{|x|=r}\left((T_{ij} - T_{ji})\nu^i\right)_{,j}\, d\mathcal{H}^{n-1}- \int_{|x|=r} (T_{ij} - T_{ji})\nu^i_{,j}\, d\mathcal{H}^{n-1}= 0.
%\end{align*}
%After relabeling the indices, it gives the first desired identity. 

%Since, for each $j$ fixed, $f_{,k} \delta_{ij} - f_{,i} \delta_{jk}$ is anti-symmetric in $i$ and $k$, we have
%\begin{align*}
%	\int_{|x|=r} (\Delta_0 f \delta_{ij} - f_{ij} ) \nu^i  \, d\mathcal{H}^{n-1}&= \int_{|x|=r}(f_k \delta_{ij} - f_i \delta_{jk})_{,k} \nu^i\, d\mathcal{H}^{n-1}=0\\
%		\int_{|x|=r} (\Delta_0 f \delta_{ij} - f_{ij} )x^j \nu^i  \, d\mathcal{H}^{n-1}&= \int_{|x|=r} \left[(f_k \delta_{ij} - f_i \delta_{jk})x^j\right]_{,k}  \nu^i\, d\mathcal{H}^{n-1}+(n-1) \int_{|x|=r} f_{,i}\nu^i\, d\mathcal{H}^{n-1}\\
%		&=  (n-1) \int_{|x|=r} f_{,i}\nu^i\, d\mathcal{H}^{n-1}. 
%\end{align*}

%\end{proof}

Throughout this section, we fix a number $q_1 \in (0, 1)$ such that 
\[
	n+ q_1\le  \min(n+q_0, 2+2q).
\] 
It will show up in the fall-off rates of error terms in many estimates.

\begin{lemma}
Let $(M, g, \pi)$ be an $n$-dimensional asymptotically flat initial data set of type $(p, q, q_0,\alpha)$,
 and let $(f,X)$ be asymptotically vacuum Killing initial data for $(g,\pi)$ that is asymptotic to some $(a,b)$, where $a\in \mathbb{R}$ and $b\in \mathbb{R}^n$. 
Then
\begin{align}
	 -( \Delta_0f) \delta_{ij} +  f_{,ij} - a R_{ij} + \tfrac{1}{2} b_k \pi_{ij, k}&= O^{0,\alpha}(|x|^{-n-q_1})\label{equation:f} \\
	X^i_{,j} + X^j_{,i}  + g_{ij,k}b_k - \tfrac{4}{n-1} a(\mathrm{tr}_0 \pi) \delta_{ij} + 4a \pi_{ij} &= O^{1,\alpha}(|x|^{-1-2q}).\label{equation:X}
\end{align}
 As a consequence, 
\begin{align}
	\Delta_0 f &= \tfrac{1}{2(n-1)} b_k (\mathrm{tr}_0 \pi)_{,k} + O^{0,\alpha}(|x|^{-n-q_1}) \label{equation:Laplacef}\\
	\Div_0X& = - \tfrac{1}{2} g_{ii,k}b_k  + \tfrac{2}{n-1} a(\mathrm{tr}_0 \pi)+ O^{1,\alpha}(|x|^{-1-2q}) \label{equation:divX}\\
	\Delta_0 X^i &=\left(\tfrac{1}{2} g_{jj,ki} - g_{ij,kj}\right) b_k+ \tfrac{2}{n-1} a(\mathrm{tr}_0\pi)_{,i} + O^{0,\alpha}(|x|^{-n-q_1}). \label{equation:LaplaceX}
\end{align}
\end{lemma}
\begin{proof}
The equations \eqref{equation:f} and \eqref{equation:X} come directly from using equation \eqref{equation:adjoint} to write out 
the statement that $D\Phi|_{(g,\pi)}^*(f, X) \in C^{0,\alpha}_{-n-q_0} \times C^{1,\alpha}_{-1-2q}$ and then using known asymptotics to simplify the expression, as well as the following equation:
\begin{align*}
	X^i_{;j} + X^j_{;i} &= X^i_{,j} + X^j_{,i} + \Gamma^i_{jk} X^k + \Gamma^j_{ik} X^k= X^i_{,j} + X^j_{,i}  + g_{ij,k}b_k + O^{1,\alpha}(|x|^{-1-2q}).
\end{align*}  
Taking the trace of \eqref{equation:f} and \eqref{equation:X} gives \eqref{equation:Laplacef} and \eqref{equation:divX}, respectively. Equation~\eqref{equation:LaplaceX} follows from differentiating \eqref{equation:X} with respect to $\partial_j$, substituting the divergence term by \eqref{equation:divX}, and using $\pi_{ij,j} \in C^{0,\alpha}_{-n-q_1}$.

\end{proof}

We can further express the next order terms in the expansion using the ADM energy-momentum.  Under the added assumption of harmonic coordinates, Beig and Chru\'{s}ciel obtained the following expansions for $f$ when $E=|P|$ and for $X$ (without the $E=|P|$ assumption)~\cite[Proofs of Proposition 3.1 and Theorem 3.4]{Beig-Chrusciel:1996}.
\begin{theorem}\label{theorem:asymptotics}
Let $(M, g, \pi)$ be asymptotically flat with  ADM energy-momentum vector $(E, P)$.  Let $(f, X)$ be asymptotically vacuum Killing initial data for $(g,\pi)$ that is  asymptotic to some $(a, b)$, where $a\in \mathbb{R}$ and $b \in \mathbb{R}^n$. Then the following expansion holds in $M\setminus K$:
\begin{align}\label{equation:expansion-alt}
\begin{split}
	f &= a + (-aE + \tfrac{1}{2(n-2)} b\cdot P) |x|^{2-n}+\tfrac{1}{2(n-1)} b_k \phi_{,k}+ O^{2,\alpha}(|x|^{2-n-q_1})\\
	X^i &=b_i -\tfrac{2(n-1)}{n-2} b_iE |x|^{2-n} + \tfrac{2}{n-1}  a\phi_{,i} + b_k V_{i,k} + O^{2,\alpha}(|x|^{2-n-q_1})
\end{split}
\end{align}
where $\phi, V_i\in C^{3,\alpha}_{1-q}$, $i=1,\dots, n$, satisfy the following equations in $M\setminus K$:
\begin{align}\label{equation:next-order}
	\begin{split}
	\Delta_0 \phi &= \mathrm{tr}_0 \pi\\
	\Delta_0 V_i &= \tfrac{1}{2} g_{jj,i} - g_{ij,j}\qquad \mbox{ for } i=1,\dots, n.
	\end{split}
\end{align}
Moreover, $b_i E =-2aP_i$. We also note 
\begin{align} \label{equation:Ricci}
	\Delta_0 (V_{i,j} + V_{j,i} + g_{ij} ) = -2 R_{ij} + O^{0,\alpha}(|x|^{-2-2q}). 
\end{align}
	
\end{theorem}

\begin{remark}
Standard elliptic theory implies that there exist $C^{3,\alpha}_{1-q}$ solutions $\phi$ and $V_i$ to~\eqref{equation:next-order} which are unique up to constant and Euclidean harmonic functions of order $|x|^{2-n}$ or lower~\cite{Meyers:1963}. Thus, the relevant terms described above in the expansion of $(f, X)$ are independent of the choices of $\phi$ and~$V_i$.
\end{remark}
\begin{remark}
Note that for the purpose of proving our main theorem (Theorem~\ref{theorem:main}), it is unnecessary to prove the second fact that $b_i E =-2aP_i$, because Theorem~\ref{theorem:nontrivial-kernel} already gives us $(f, X)$ with $(a,b)=(E, -2P)$. However, it is interesting to note that the proportionality must hold  more generally. 
\end{remark}

\begin{proof}
Let $\phi$ and  $V_i$ solve \eqref{equation:next-order}. These quantities are chosen so that their Laplacians exactly match the non-homogenous terms of \eqref{equation:Laplacef} and \eqref{equation:LaplaceX}. Therefore harmonic expansion (see e.g. \cite{Meyers:1963}) tells us that there are constants $A, B_i$ such that 
\begin{align} 
	f &= a +A|x|^{2-n} + \tfrac{1}{2(n-1)} b_k  \phi_{,k} + O^{2,\alpha}(|x|^{2-n-q_1}) \label{equation:asymptoticsf}\\
	X^i &= b_i + B_i |x|^{2-n} + \tfrac{2}{n-1}a\phi_{,i}  + b_k V_{i,k} + O^{2,\alpha}(|x|^{2-n-q_1}).\label{equation:asymptoticsX}
\end{align}
The limitation of this expansion comes from the fact that we do not expect the $\phi$ and $V_i$ terms appearing in the expansion to be lower order than $|x|^{2-n}$. However, in what follows, we see that we are able to handle them.

We will  establish \eqref{equation:expansion-alt} by  showing that 
\begin{align}
	A&= -aE + \tfrac{1}{2(n-2)} b\cdot P \label{equation:A}\\
	B_i &=-\tfrac{2(n-1)}{n-2} b_i E \label{equation:B}.
	%B_i &= -\tfrac{2(n-1)}{n-2} b_i E 
\end{align}
We first prove \eqref{equation:A}. Consider equation~\eqref{equation:f}:
\[  -( \Delta_0f) \delta_{ij} +  f_{,ij} - a R_{ij} + \tfrac{1}{2} b_k \pi_{ij, k}= O^{0,\alpha}(|x|^{-n-q_1}).\]
It is well-known that we can express $E$ as a flux integral involving the Ricci curvature (see, for example, \cite{Huang:2010, Miao-Tam:2016}) and thus
\[	
\int_{S_\infty} -aR_{ij}x_i\nu_j \, d\mathcal{H}^{n-1} = (n-1)(n-2)\omega_{n-1} aE.
\]
This suggests that we should integrate  equation~\eqref{equation:f} against $x_i\nu_j$ over $S_\infty$. %Because of~\eqref{equation:asymptoticsf}, we expect to see $A$ to arise from the flux integrals of the $f$ terms, and we expect a $P$ term to arise from integrating the $\pi$ term. Because of \eqref{equation:asymptoticsf}, we also expect integrals of $\phi$ terms to arise, but it turns out that these terms cancel. We compute each integral separately. 
By the second identity of Corollary~\ref{corollary:basic} and equation~\eqref{equation:asymptoticsf}, we see that 
\begin{align*}
&\int_{S_\infty} \left[ -( \Delta_0f) \delta_{ij} +  f_{,ij} \right]x_i\nu_j \, d\mathcal{H}^{n-1}\\
	& = (1-n)\int_{S_\infty}  f_{,i} \nu_i\, d\mathcal{H}^{n-1} \\
	& = (1-n) \int_{S_\infty} \left[ (2-n) A |x|^{-n} x_i +\tfrac{1}{2(n-1)} b_k\phi_{,ki} \right] \nu_i\, d\mathcal{H}^{n-1} \\
	&=(1-n) (2-n)\omega_{n-1}A  -\tfrac{1}{2}\int_{|x|=r} (b\cdot \nu) \Delta_0 \phi \, d\mathcal{H}^{n-1} \quad \mbox{(by Corollary~\ref{corollary:basic})}\\
	&=(n-1)(n-2) \omega_{n-1} A -\tfrac{1}{2} \int_{|x|=r} (b\cdot \nu) \mathrm{tr}_0 \pi\, d\mathcal{H}^{n-1}.
\end{align*}
To compute the last flux integral from \eqref{equation:f}, we apply Lemma~\ref{lemma:basic} for the tensor $T_{jk} = b_k \pi_{ij} x_i$ in the second equality below  to obtain
\begin{align*}
	 \tfrac{1}{2} \int_{S_\infty}b_k \pi_{ij,k} x_i\nu_j\, d\mathcal{H}^{n-1}  &= \tfrac{1}{2}  \int_{S_\infty} \left[(b_k\pi_{ij} x_i)_{,k} \nu_j- b_k \pi_{kj} \nu_j\right] \, d\mathcal{H}^{n-1}\\
	&= \tfrac{1}{2} \int_{S_\infty}  \left[(b_j\pi_{ik} x_i)_{,k} \nu_j- b_k \pi_{kj} \nu_j\right] \, d\mathcal{H}^{n-1}\\
	&=  \tfrac{1}{2}  \int_{S_\infty}\left[ b_j \pi_{ik,k} x_i\nu_j + (b\cdot \nu) \mathrm{tr}_0\pi - b_k \pi_{kj}\nu_j\right]\, d\mathcal{H}^{n-1}\\
	&=  \tfrac{1}{2} \int_{S_\infty} (b\cdot \nu) \mathrm{tr}_0 \pi \, d\mathcal{H}^{n-1} -  \tfrac{n-1}{2} \omega_{n-1} b\cdot P
\end{align*}
where in the last equality we used the definition of $P$ and the fact that $\pi_{ik,k} = O(|x|^{-n-q_1})$, so the corresponding term integrates to zero in the limit. Knowing that the three previous computations must add up to zero, we obtain
\[ 0 =  (n-1)(n-2)\omega_{n-1} aE + (n-1)(n-2) \omega_{n-1} A - \tfrac{n-1}{2} \omega_{n-1} b\cdot P,\]
which establishes equation~\eqref{equation:A}.

In what follows, we will need the asymptotic expansion of $\Div_0 V$. Observe \eqref{equation:Ricci}:
\[
	\Delta_0 (V_{i,j} + V_{j,i} + g_{ij} ) = g_{kk,ij} + g_{ij,kk}- g_{ik,kj} -g_{jk,ji} = -2 R_{ij} + O^{0,\alpha}(|x|^{-2-2q}).  
\]
Taking the trace of the equation and  using harmonic expansion and $R_g = O^{0,\alpha}(|x|^{-n-q_0})$, we derive
 \begin{equation} \label{equation:divV}
 \Div_0 V =  \tfrac{1}{2} (n-g_{ii}) + \beta |x|^{2-n} + O^{2,\alpha}(|x|^{2-n-q_1}),
 \end{equation}
 for some constant~$\beta$.  We compute $\beta$ by computing the flux of $\Div_0 V$ in two ways. First, using the expansion~\eqref{equation:divV}, 
\begin{align*}
\int_{S_\infty}  (\Div_0 V)_{,j}  \nu_j\, d\mathcal{H}^{n-1}
 &=\int_{S_\infty} \left(-\tfrac{1}{2} g_{ii,j} + (2-n) \beta x_j |x|^{-n}\right)\nu_j\, d\mathcal{H}^{n-1} \\
 &=\int_{S_\infty} -\tfrac{1}{2} g_{ii,j}\nu_j\, d\mathcal{H}^{n-1} + (2-n)\omega_{n-1}\beta.
 \end{align*}
 Second, we use Corollary~\ref{corollary:basic} and the definition of $V_i$ from~\eqref{equation:next-order} to find
\begin{align*}
\int_{S_\infty}  (\Div_0 V)_{,j}  \nu_j\, d\mathcal{H}^{n-1}
= \int_{S_\infty} (\Delta_0 V_i) \nu_i \, d\mathcal{H}^{n-1}=	\int_{S_\infty}\left( \tfrac{1}{2} g_{jj,i} - g_{ij,j}\right)\nu_i \, d\mathcal{H}^{n-1}.
\end{align*}
Thus 
\[
\beta= \tfrac{1}{(n-2)\omega_{n-1}} \int_{S_\infty} ( g_{ij,j}-g_{jj,i} ) \nu_i \, d\mathcal{H}^{n-1}=  \tfrac{2(n-1)}{n-2} E.
\]

 Next we will prove $B_i = -\tfrac{2(n-1)}{n-2}b_i E$. Recall equation~\eqref{equation:divX}:
\[
	\Div_0 X = - \tfrac{1}{2} g_{ii,k}b_k  + \tfrac{2}{n-1} a(\mathrm{tr}_0 \pi)+ O^{1,\alpha}(|x|^{-1-2q}).
\]
We can also compute the divergence using the expansion for $X$ in~\eqref{equation:asymptoticsX}:
\begin{equation}\label{equation:divX2}
\Div_0 X = (2-n)B_i |x|^{-n}x_i + 
 \tfrac{2}{n-1}a\Delta_0\phi  + b_k (\Div_0 V)_{,k} + O^{1,\alpha}(|x|^{1-n-q_1}).
\end{equation}
By comparing these two equations, the definition of $\phi$, and our expansion of $\Div_0 V$ in~\eqref{equation:divV}, we obtain 
\begin{align*}
 - \tfrac{1}{2} g_{ii,k}b_k 
 &=(2-n)B_i |x|^{-n}x_i  + b_k (\Div_0 V)_{,k} + O^{1,\alpha}(|x|^{1-n-q_1}) \\
 &= (2-n)B_i |x|^{-n}x_i   -\tfrac{1}{2} g_{ii,k} b_k+(2-n)b_k \beta |x|^{-n}x_k+ O^{1,\alpha}(|x|^{1-n-q_1}).
\end{align*}
 Thus $B_i = -b_i\beta = -\tfrac{2(n-1)}{n-2} b_i E$.

Finally we will prove that $b_i E = -2aP_i$ by showing that  $B_i = \tfrac{4(n-1)}{n-2} aP_i$ in a similar manner to how we proved equation~\eqref{equation:A} and combining this with our previous formula for $B_i$. Consider the equation~\eqref{equation:X}: 
\[
X^i_{,j} + X^j_{,i}  + g_{ij,k}b_k - \tfrac{4}{n-1} a(\mathrm{tr}_0 \pi) \delta_{ij} + 4a \pi_{ij} = O^{1,\alpha}(|x|^{-1-2q}). 
\]
As before, we will use the fact that the flux integral of the above quantity must be zero. We know that the flux of the last term is 
\[ \int_{S_\infty} 4a \pi_{ij} \nu_j \, d\mathcal{H}^{n-1} = 4(n-1)\omega_{n-1}a P_i,\]
and we expect $B_i$ to show up when we take the flux of the $X$ terms. 
Using the expansion for $X$, as well as Lemma~\ref{lemma:basic} (with $T_{jk}=b_k(V_{i,j}+V_{j,i}) $ in the second equality and with $T_{jk} = b_j g_{ik}$ in the last equality) and Corollary~\ref{corollary:basic} liberally,
\begin{align*}
&\int_{S_\infty} (X^i_{,j}+X^j_{,i})\nu_j \, d\mathcal{H}^{n-1} \\
&= \int_{S_\infty}\left[  (2-n)|x|^{-n}( B_i x_j + B_j x_i) + \tfrac{4}{n-1}a \phi_{,ij} +
b_k (V_{i,kj} + V_{j,ki} )\right]   \nu_j \, d\mathcal{H}^{n-1} \\
&=(2-n)\omega_{n-1}B_i + \int_{S_\infty}\left[  (2-n) |x|^{-n} B_j x_i+ \tfrac{4}{n-1}a \Delta_0 \phi \delta_{ij}+
b_j (\Delta_0 V_i +  (\Div_0 V)_{,i} )\right]   \nu_j \, d\mathcal{H}^{n-1} \\
&=(2-n)\omega_{n-1}B_i + \int_{S_\infty}\left[ (2-n) |x|^{-n} B_j x_i+  \tfrac{4}{n-1}a \Delta_0 \phi \delta_{ij}
-b_jg_{ik,k} + (2-n) |x|^{-n}b_j\beta x_i )\right]   \nu_j \, d\mathcal{H}^{n-1}\\
&=(2-n)\omega_{n-1}B_i + \int_{S_\infty}\left[   \tfrac{4}{n-1}a (\mathrm{tr}_0\pi) \delta_{ij}
-b_k g_{ij,k}\right]   \nu_j \, d\mathcal{H}^{n-1},
\end{align*}
where we use $B_j = -b_j \beta$ in the last equality.
Now it is apparent that the flux integrals of the terms $g_{ij,k}b_k - \tfrac{4}{n-1} a(\mathrm{tr}_0 \pi) \delta_{ij}$ from \eqref{equation:X} will cancel against integrals in the above expression. Putting it all together, we obtain the desired equation
 $B_i = \tfrac{4(n-1)}{n-2}a P_i$.

\end{proof}

We immediately obtain the following corollary.
\begin{corollary}\label{corollary:asymptotics}
Under the same assumption as in Theorem~\ref{theorem:asymptotics}, we have the following:
\begin{enumerate}
\item If $E\neq 0$ and $a\neq 0$, then $(a,b)$ is proportional to $(E, -2P)$, and thus, up to scaling, we have
\begin{align}\label{equation:expansion}
\begin{split}
	f &= E -\left (E^2 + \tfrac{1}{n-2} |P|^2\right) |x|^{2-n}-\tfrac{1}{n-1} P_k \phi_{,k}+ O^{2,\alpha}(|x|^{2-n-q_1})\\
	X^i &=-2P_i + \tfrac{4(n-1)}{n-2} EP_i |x|^{2-n} + \tfrac{2}{n-1} E \phi_{,i} - 2P_k V_{i,k} + O^{2,\alpha}(|x|^{2-n-q_1}).
\end{split}
\end{align}
\item If $E\neq 0$ and $a=0$, then $b=0$. 
\item If $E=0$, then either $a=0$ or $P=0$, and $(f,X)$ satisfies 
\begin{align*}
	f &= a+\tfrac{1}{2(n-2)} b\cdot P |x|^{2-n} + \tfrac{1}{2(n-1)} b_k \phi_{,k}+ O^{2,\alpha}(|x|^{2-n-q_1})\\
	X^i &=b_i + \tfrac{2}{n-1} a \phi_{,i} + b_k V_{i,k} + O^{2,\alpha}(|x|^{2-n-q_1}).
\end{align*}
\end{enumerate}
\end{corollary}

\begin{proof}[Proof of Theorem \ref{theorem:zero}]
%As mentioned above, the main difference between our proof and the original one in \cite{Beig-Chrusciel:1996} is that we derive the asymptotics of $(f,X)$ in greater generality as shown in Theorem~\ref{theorem:asymptotics}, and we avoid the use of harmonic coordinates. We also slightly relax the regularity assumption.

We begin by assuming that $(f,X)$ is asymptotically vacuum Killing initial data for $(g,\pi)$ that is asymptotic to some $(a,b)$, where $a\in \mathbb{R}$ and $b\in \mathbb{R}^n$ are not all zero. Suppose that $E\neq0$. By  Corollary~\ref{corollary:asymptotics}, it follows that $a\neq 0$ and we can scale  $(f, X)$ so that $(f,X)$ is asymptotic to $(E,-2P)$. We can also rotate our coordinates so that without loss of generality, $P$ points in the $x_n$-direction. That is, $P=(0,\ldots,0,|P|)$.

Now substitute what we know about $(a,b)$ into \eqref{equation:f} and \eqref{equation:X} and also replace
the $\Delta_0 f$ term using \eqref{equation:Laplacef}. Doing this we obtain
\begin{align}
	 \tfrac{1}{n-1}|P| (\mathrm{tr}_0 \pi)_{,n} \delta_{ij} + f_{,ij} - ER_{ij} - |P|\pi_{ij,n} &= O^{0,\alpha}(|x|^{-n-q_1})\label{equation:ff}\\
	 X^i_{,j} + X^j_{,i} - 2|P| g_{ij,n} - \tfrac{4}{n-1} E(\mathrm{tr}_0\pi) \delta_{ij}+4E\pi_{ij} &=  O^{1,\alpha}(|x|^{-1-2q}).\label{equation:XX}
\end{align}
%Let $\Delta'$ denote the Euclidean Laplace operator in the first $(n-1)$ coordinates $x_1,\dots, x_{n-1}$. Looking at the top order terms of the Ricci curvature tensor and then using the definition of $V$, we have
%\begin{align*}
%	2R_{ij} &=  g_{ik,jk} + g_{jk,ik} - g_{ij,kk} - g_{kk,ij} + O^{0,\alpha} (|x|^{-2-2q})\\
%	&= -  \Delta' g_{ij} -  g_{ij,nn} +  (g_{ik,k}-\tfrac{1}{2} g_{kk,i})_{,j} +  (g_{jk,k} - \tfrac{1}{2}  g_{kk,j})_{,i} + O^{0,\alpha} (|x|^{-2-2q})\\
%	&= -  \Delta' g_{ij} -  g_{ij,nn} - \Delta_0 (V_{i,j} + V_{j,i})+ O^{0,\alpha} (|x|^{-2-2q}).
%\end{align*}
%Substituting this expression into \eqref{equation:ff},
%\begin{align*}
 % \tfrac{1}{n-1}|P| (\mathrm{tr}_0 \pi)_{,n} \delta_{ij} + f_{,ij} 
 %  + \tfrac{1}{2} E\Delta' g_{ij} +\tfrac{1}{2} E\Delta_0 ( V_{i, j} + V_{j, i})+ \tfrac{1}{2} Eg_{ij,nn} 
  % - |P|\pi_{ij,n}  =	  O^{0,\alpha}(|x|^{-n-q_1}).
%\end{align*}
%Expanding the $f_{,ij}$ in the above equation using~\eqref{equation:f-X}, rearranging some terms, and multiplying by~$4$, we obtain
%\begin{multline}\label{equation:Laplace-g}
% 2E\Delta' g_{ij} -4\left (E^2 + \tfrac{1}{n-2} |P|^2\right)\partial_i\partial_j |x|^{2-n}
 % -\tfrac{4}{n-1}|P|\phi_{,ijn}  \\
% +2E\Delta_0 (V_{i,j} + V_{j, i})+  \tfrac{4}{n-1}|P| (\mathrm{tr}_0 \pi)_{,n} \delta_{ij} + 2Eg_{ij,nn} 
 %  - 4|P|\pi_{ij,n}  =	  O^{0,\alpha}(|x|^{-n-q_1}).
%\end{multline}
Differentiate~\eqref{equation:XX} in the $x_n$-direction to obtain 
\begin{align} \label{equation:X_n}
 X^i_{,jn} + X^j_{,in} - 2|P| g_{ij,nn} - \tfrac{4}{n-1} E(\mathrm{tr}_0\pi)_{,n} \delta_{ij}+4E\pi_{ij,n} =  O^{0,\alpha}(|x|^{-2-2q}).
 \end{align}
%We  use the asymptotics of $X$ in~\eqref{equation:f-X} in the above equation and rearrange some terms to obtain
%\begin{multline}\notag
%	 \tfrac{4(n-1)}{n-2} E|P| (\partial_j \partial_n |x|^{2-n}\delta_{in} + \partial_i \partial_n |x|^{2-n} \delta_{jn}) + \tfrac{4}{n-1} E \phi_{,ijn} \\
%	 - 2|P| (V_{i,j}+ V_{j,i})_{,nn}- \tfrac{4}{n-1} E(\mathrm{tr}_0 \pi)_{,n}\delta_{ij} - 2|P| g_{ij,nn} + 4E \pi_{ij,n} =	  O^{0,\alpha}(|x|^{-n-q_1}).
%\end{multline}
(Note that $n$ is fixed and not a summation index.) Equations \eqref{equation:ff} and \eqref{equation:X_n} will combine very nicely precisely when $E=|P|$. So from now on we invoke the hypothesis that $E=|P|$. Combining those two equations together, we obtain
\begin{align}\label{equation:combination}
	4f_{,ij} - 4E R_{ij} + X^i_{,jn} + X^j_{,in} - 2E g_{ij,nn} = O^{0,\alpha}(|x|^{-n-q_1}). 
\end{align}

By Corollary~\ref{corollary:asymptotics}, equations \eqref{equation:expansion} hold, and they now reduce to
\begin{align}
\begin{split}\label{equation:f-X}
	f &= E -\tfrac{n-1}{n-2} E^2 |x|^{2-n}-\tfrac{1}{n-1}E \phi_{,n}+ O^{2,\alpha}(|x|^{2-n-q_1})\\
	X^i &=-2E\delta_{in} + \tfrac{4(n-1)}{n-2} E^2 |x|^{2-n} \delta_{in}+ \tfrac{2}{n-1} E \phi_{,i} - 2E V_{i,n} + O^{2,\alpha}(|x|^{2-n-q_1}).
\end{split}
\end{align}
Substitute the asymptotics of $(f, X)$ into \eqref{equation:combination} and replace the Ricci term by \eqref{equation:Ricci}. We obtain
\begin{align*}
 &2E\Delta' (g_{ij} +V_{i,j}+V_{i,j})\\
 &\quad  = \tfrac{4(n-1)}{n-2}E^2 \left( \partial_i\partial_j |x|^{2-n} -  (\partial_j \partial_n |x|^{2-n}\delta_{in} + \partial_i \partial_n |x|^{2-n} \delta_{jn}) \right)+ O^{0,\alpha}(|x|^{-n-q_1})
\end{align*}
where $\Delta'$ denotes Euclidean Laplacian in the first $(n-1)$ coordinates $(x_1, \dots, x_{n-1})$. 
 
We will use capital letters to denote indices running from $1$ to $n-1$, $x' = (x_1,\dots, x_{n-1})$, and $\rho = |x'|$. If we define 
\[ 
\omega_{AB} = g_{AB}-\delta_{AB}+ V_{A,B}+V_{B,A},
\]
then  $\omega_{AB} \in C^{2,\alpha}_{-q} (M\setminus K)$ and
\[
\Delta' \omega_{AB} = \tfrac{2(n-1)}{n-2}E \partial_A\partial_B |x|^{2-n}+ O^{0,\alpha}(|x|^{-n-q_1}).
\]

Now define
\[
	Y_B = x_Ax_C w_{AC, B} -\tfrac{1}{n-1} \rho^2 w_{AA, B} - 2 x_A w_{AB} + \tfrac{2}{n-1} x_B w_{AA},
\]
so that $Y_B\in  C^{1,\alpha}_{1-q} (M\setminus K)$ and $Y_n = 0$.
Denoting the divergence operator of the first $(n-1)$ components by $\Div'$, we can  compute
\begin{align}\label{equation:Y}
\begin{split}
	\Div' Y &= x_A x_B \Delta' \omega_{AB} - \tfrac{1}{n-1} \rho^2 \Delta' \omega_{AA}\\
 &=2(n-1)E(n\rho^4 |x|^{-n-2} - \rho^2 |x|^{-n})- 2E (n\rho^4 |x|^{-n-2} -(n-1) \rho^2 |x|^{-n})+ O^{0,\alpha}(|x|^{2-n-q_1})\\
 &=-2E\rho^4 |x|^{-n-2} + O^{0,\alpha}(|x|^{2-n-q_1}).
\end{split}
\end{align}
As a matter of pure analysis, if $\partial_n Y$ decays sufficiently fast, this is impossible unless $E=0$. This completes the proof, modulo the technical lemma immediately below.
\end{proof}

The following lemma is the only place where the assumption~\eqref{equation:extra} that $q+ \alpha > n-2$ is required. 
\begin{lemma}\label{lemma:strong-decay}
Let $q$ and $\alpha$ be  numbers such that $\alpha\in (0, 1)$ and $q+ \alpha > n-2$. Let $Y\in C^{1,\alpha}_{1-q}$ be a vector field on $\mathbb{R}^n\setminus B$. Suppose that $Y$ satisfies
\[
	\Div'Y = -2E\rho^4 |x|^{-n-2} + v(x)
\]
where $E$ is a constant and $v(x) \in C^{0,\alpha}_{2-n-q_1} (\mathbb{R}^n\setminus B) $ for some $q_1>0$. Then $E=0$. 
\end{lemma}
\begin{proof}
%The proof is to show that the divergence equation with nonzero $E$ is incompatible with the fall-off rate of $\partial_n Y$. Using the vector field $Y$ to define a flux integral on the $(n-2)$ sphere $|x'|=\rho$ on each $x_n$-slice, we can compute how it depends on $x_n$ using the divergence equation. On the other hand, the fall-off rate of $Y$ implies that the flux integral converges to a constant independent of $x_n$ as $\rho\to \infty$ and leads a contradiction, unless $E=0$.

Suppose on the contrary that  $E\neq 0$. We may assume $E>0$. For each $x_n$ and $h>0$, define the limit of flux integrals on each $x_n$-slice by
\[
	I_h(x_n)=\lim_{\rho\to\infty}\int_{|x'|=\rho}D_hY(x', x_n) \cdot \frac{x'}{\rho}\, d\mathcal{H}^{n-2},
\]	
where $D_h Y$ is the difference quotient in the $x_n$ coordinate defined by, for each $h> 0$,
\[
	(D_h Y) (x', x_n) = \frac{Y(x', x_n+h) - Y(x', x_n)}{ h}.
\]
We will choose $h$ to depend on $x_n$ later, so we will drop the subscript $h$ of $I_h$ to simplify the notation. We note that if $Y$ has stronger regularity, e.g. $Y\in C^{2,\alpha}_{1-q}$, then we can use $\partial_n Y$ as in \cite{Beig-Chrusciel:1996}, instead of the delicate difference quotient. 

We now compute the limit. By divergence theorem on the $x_n$-slice, we have
\[	
 	I(x_n)=\int_{\mathbb{R}^{n-1}} \Div'D_h Y\, dx' = \int_{\mathbb{R}^{n-1}} D_h (\Div'Y)\, dx'.
\]
We denote by $u(x', x_n) = -2E\rho^4 |x|^{-n-2}$. By Taylor expansion in the $x_n$ coordinate,   
\[
  D_hu(x', x_n) = 2(n+2) E \rho^4 |x|^{-n-4} x_n+ O(\rho^4|x|^{-n-4}h).
 \]
For the $v$ term, we have 
 \[
	|D_h v(x', x_n)|\le [v]_{\alpha} h^{-1+\alpha}\le \| v\|_{C^{0,\alpha}_{2-n-q_1}} |x|^{2-n-q_1-\alpha} h^{-1+\alpha}.
\]
Combining the above computations, we can rewrite the integrand as 
\[
	D_h (\Div'Y) = 2(n+2) E \rho^4 |x|^{-n-4} x_n+ O(\rho^4|x|^{-n-4}h + |x|^{2-n-q_1-\alpha} h^{-1+\alpha}).
\]
In order for the $E$ term to dominate, we choose $h=x_n^{2s}$ where $s>0$ satisfies
\[
	1 - \frac{q_1}{(1-\alpha)} < 2s < 1.
\]

We will use the fact that for any positive real numbers $a, b$ with $b-a < 1-n$, there exists constants $0< C_1 < C_2$ depending only on $n, a, b$ such that
\[
	 C_1|x_n|^{n-1-a+b} \le \int_{\mathbb{R}^{n-1}}\rho^b |x|^{-a} \, d x'\le C_2|x_n|^{n-1-a+b}.
\]
The proof is a straightforward computation by estimating the integral over the regions where  $\rho\le |x_n|$ and $\rho\ge |x_n|$  separately. Combining the above inequalities with the equation of $D_h (\Div'Y)$  allows us to estimate $I(x_n)$  as follows, for some constant $C$ independent of $x_n$:
\begin{align*}
	I(x_n) &\ge 2(n+2) C_1 E - C(|x_n|^{-1+2s} + |x_n|^{1-q_1-\alpha-2s(1-\alpha)}) &&\mbox{ if } x_n>0\\
	I(x_n) &\le - 2(n+2) C_1 E + C(|x_n|^{-1+2s} + |x_n|^{1-q_1-\alpha - 2s(1-\alpha)})&&\mbox{ if } x_n<0.
\end{align*}
Our hypothesis on $s$ implies that the $E$ term dominates, and hence, for $|x_n|$ sufficiently large, we have $I(x_n) >0$ if $x_n>0$ and  $I(x_n)<0$ if $x_n<0$.

On the other hand, this will contradict the decay assumption of $Y$ as follows. For every $\bar{h}$,
\begin{align*}
	I(\bar{h}) - I(0)&=\lim_{\rho \to \infty}  \int_{\{|x'|=\rho\}} \int_0^{\bar{h}}\partial_n(D_{x_n^{2s}} Y)\cdot \frac{x'}{\rho} \, dx_nd\mathcal{H}^{n-2}.
\end{align*}
Computing the integrand, for $|x_n|>0$, 
\begin{align*}
	&\partial_n(D_{x_n^{2s}} Y)\\
	&=\partial_n \left[ \frac{Y(x', x_n + x_n^{2s}) - Y(x', x_n) }{x_n^{2s}}\right]\\
	&=\frac{(\partial_nY)(x', x_n+ x_n^{2s}) - (\partial_n Y)(x' , x_n)}{x_n^{2s}} + \frac{2s}{x_n} \left[(\partial_n Y)(x', x_n + x_n^{2s}) - \frac{Y(x', x_n+x_n^{2s}) - Y(x', x_n)}{x_n^{2s}}\right]\\
	&=\frac{(\partial_nY)(x', x_n+ x_n^{2s}) - (\partial_n Y)(x' , x_n)}{x_n^{2s}} + \frac{2s}{x_n}\left[(\partial_n Y)(x', x_n + x_n^{2s}) - (\partial_n Y)(x', x_n + c) \right]
\end{align*}
for some $c\in (0, x_n^{2s})$ by Mean Value Theorem. Then we estimate term by term as follows:
\begin{align*}
	\left|\partial_n(D_{x_n^{2s}} Y)\right|&\le [ \partial_n Y]_\alpha |x_n|^{2s(\alpha-1)} + \frac{2s}{|x_n|} [\partial_n Y]_\alpha |x_n^{2s} - c|^\alpha\\
	& \le \left(|x_n|^{2s(\alpha-1)}+2s|x_n|^{-1+2s\alpha}\right)\| \partial_n Y\|_{C^{0,\alpha}_{-q}(\mathbb{R}^n\setminus B)}|x|^{-q-\alpha}. 
\end{align*}
Our assumption $q+\alpha > n-2$, as well as $2s<1$,  implies that
\[
	| I(\bar{h}) - I(0) |\le \omega_{n-2}\| \partial_n Y\|_{C^{0,\alpha}_{-q}(\mathbb{R}^n\setminus B)} \lim_{\rho \to \infty} \rho^{-q-\alpha} \rho^{n-2} \int_0^{\bar{h}} (|x_n|^{2s(\alpha-1)}+2s|x_n|^{-1+2s\alpha})  dx_n=0.
\]

\end{proof}

\section{Elliptic regularity and the adjoint equation}\label{appendix:regularity}

We begin with a version of elliptic regularity for $L^p$ weak solutions. Since we were unable to find a direct reference that applies in our setting, we include the proof.

\begin{proposition} \label{proposition:Lp-estimate}
 Let $U$ be a bounded open subset of $\mathbb{R}^n$ and let $u\in L^p(U)$, $p\in (1, \infty)$, be a weak solution to \[
	\Delta_0 u+b_i \partial_i u+ cu = \Div V +f,
\] 
where $\Delta_0$ is the Euclidean Laplace operator, $b = (b_1,\dots, b_n) \in W^{1,\infty}(U), c\in L^\infty(U)$, and $f\in L^p(U)$.  
\begin{enumerate}
\item If $V=(V_1, \dots, V_n)\in L^p(U)$, then $u\in W^{1,p}_{\mathrm{loc}}(U)$ and for any $U'\subset\subset U$, 
\[
	\| u\|_{W^{1,p}(U')}\le C(\| u \|_{L^p(U)} + \| V \|_{L^p(U)} + \| f \|_{L^p(U)}),
\] 
where $C$ depends on $n, p, \| b\|_{W^{1,\infty}(U)}, \| c \|_{L^{\infty}(U)},  U', U$.
\item \label{item:Wp-regularity} If $V\in W^{1,p}(U)$, then $u\in W^{2,p}_{\mathrm{loc}}(U)$  and for any $U'\subset\subset U$,  
\[
	\| u\|_{W^{2,p}(U')}\le C(\| u \|_{L^p(U)} + \| V \|_{W^{1,p}(U)} + \| f \|_{L^p(U)}),
\]
where $C$ depends on $n, p,  \| b\|_{W^{1,\infty}(U)}, \| c \|_{L^{\infty}(U)}, U', U$.
\end{enumerate}
\end{proposition}
\begin{proof}
The general case can be reduced to the $b=0$, $c=0$ case by expressing 
\[
	\Delta_0 u =  (\Div b - c) u + f+\Div (V- bu).
\]	
So from here on, we assume without loss of generality that $b=0$ and $c=0$.

For $\epsilon>0$, we denote by $\phi_\epsilon$ the mollification of  $\phi$. Then
\[
	\Delta_0 u_{\epsilon} = \Div V_\epsilon+ f_{\epsilon}.
\]
Let $\chi$ be a smooth bump function that is identically  $1$ in $U'$ and supported in $W\subset\subset U$. We compute
\begin{align}\label{equation:cut-off}
	\Delta_0  (\chi u_\epsilon) &=\Div ( 2 u_\epsilon \nabla \chi + \chi V_\epsilon) - u_\epsilon \Delta_0 \chi - \nabla \chi \cdot V_\epsilon+ \chi f_\epsilon= \Div \overline{V} + \overline{f},
\end{align}
where  $\overline{V} := 2 u_\epsilon \nabla \chi + \chi V_\epsilon$ and $\overline{f} :=  - u_\epsilon \Delta_0\chi - \nabla \chi \cdot V_\epsilon+ \chi f_\epsilon$. Since $\chi u_\epsilon$ is smooth and compactly supported in $U$, by uniqueness of solutions, we have
 \[\chi u_\epsilon = v+w,\]
  where $v$ and $w$ are the Newtonian potentials of $\Div\overline{V}$ and $\overline{f}$ respectively. By the Calder\'{o}n-Zygmund inequality (see~\cite[Theorem~9.9]{Gilbarg-Trudinger:1983}, for example) and the Poincar\'e  inequality, we have
\begin{align*}
	\| v \|_{W^{1,p}(W)}\le C \| \overline{V}\|_{L^p(W)}\\
	\| w \|_{W^{2,p}(W)}\le C \| \overline{f}\|_{L^p(W)},
\end{align*}
 where $C$ depends only on $n, p, |W|$. Thus, we derive, for $\epsilon>0$ sufficiently small,
\begin{align}\label{equation:epsilon}
\begin{split}
	\|  u_\epsilon\|_{W^{1,p}(U') } \le \| \chi u_\epsilon\|_{W^{1,p}(W) }  &\le (\| v \|_{W^{1,p}(W)} + \| w \|_{W^{1,p}(W)})\\
	&\le C ( \| \overline{V} \|_{L^p(W)} +\|\overline{f} \|_{L^p(W)} )\\
	&\le C( \| u\|_{L^p(U)} + \| V\|_{L^p(U)}+\| f \|_{L^p(U)}),
\end{split}
\end{align}
where we enlarge the constant $C$ if needed and, in the last inequality we use the fact that mollifications satisfy $\| \phi_\epsilon\|_{L^p(W)} \le  \| \phi\|_{L^p(U)}$, provided  $\epsilon>0$ is sufficiently small. By the Rellich-Kondrachov compactness theorem, a subsequence of $u_{\epsilon}$ weakly converges to  a limit in $W^{1,p}(U')$. By uniqueness of the weak limit, the limit must be~$u$.  Therefore $u\in W^{1,p}(U')$ and hence $u_\epsilon \to u$ in $W^{1,p}(U')$. Passing $\epsilon\to 0$ in the previous estimate gives the desired result. 

If we further assume that $V\in W^{1,p}(U)$, then the right hand side of \eqref{equation:cut-off}  has  better regularity. Following a similar argument and using the derived $W^{1,p}$ estimate for $u$, we obtain the desired $W^{2,p}$ estimate. 
\end{proof}

Proposition~\ref{proposition:Lp-estimate} implies the following regularity for weighted norms by a standard scaling technique.

\begin{corollary}\label{corollary:Lp-estimate}
Let $\delta \in \mathbb{R}$ and $p\in (1, \infty)$. Let $u\in L^p_\delta (\mathbb{R}^n)$ be a weak solution to 
\[
	\Delta_0 u+b_i \partial_i u+ cu = \Div V +f,
\] 
where the coefficients  $b = (b_1,\dots, b_n) \in W^{1,\infty}_{-1}  (\mathbb{R}^n), c\in L^\infty_{-2} (\mathbb{R}^n)$, and $f\in L^p_{-2+\delta} (\mathbb{R}^n)$. 
\begin{enumerate}
\item If $V=(V_1, \dots, V_n)\in L^p_{-1+\delta} (\mathbb{R}^n)$, then $u\in W^{1,p}_{\delta} (\mathbb{R}^n)$ and 
\[
	\| u\|_{W^{1,p}_{\delta} (\mathbb{R}^n)}\le C(\| u \|_{L^p_\delta (\mathbb{R}^n)} + \| V \|_{L^p_{-1+\delta}  (\mathbb{R}^n)} + \| f \|_{L^p_{-2+\delta} (\mathbb{R}^n)}),
\] 
where $C$ depends on $n, p, \delta, \| b\|_{W^{1,\infty}_{-1} (\mathbb{R}^n)}, \| c \|_{L^{\infty}_{-2} (\mathbb{R}^n)}$.
\item If $V\in W^{1,p}_{-1+\delta} (\mathbb{R}^n)$, then $u\in W^{2,p}_{\delta} (\mathbb{R}^n)$  and  
\[
	\| u\|_{W^{2,p}_\delta  (\mathbb{R}^n)}\le C(\| u \|_{L^p_\delta (\mathbb{R}^n)} + \| V \|_{W^{1,p}_{-1+\delta}  (\mathbb{R}^n)} + \| f \|_{L^p_{-2+\delta} (\mathbb{R}^n)}),
\]
where $C$ depends on $n, p, \delta, \| b\|_{W^{1,\infty}_{-1}(\mathbb{R}^n)}, \| c \|_{L^{\infty}_{-2}(\mathbb{R}^n)}$.
\end{enumerate}
\end{corollary}

The adjoint operator $(D\Phi_{(g,\pi)}) ^*$ gives rise to an over-determined elliptic system, and the solutions enjoy elliptic regularity that we will discuss below. Let $(M, g,\pi)$ be an $n$-dimensional initial data set. Recall the formal $L^2$ adjoint operator of the linearized modified constraint operator:
\begin{align} \tag{\ref{equation:modified-adjoint}}
\begin{split}
 	(D\overline{\Phi}_{(g,\pi)}) ^*(f, X)  & = \left(  L_g^*f +\left( \tfrac{2}{n-1} (\mbox{tr}_g \pi) \pi_{ij} - 2 \pi_{ik} \pi^k_j \right) f\right.
	\\
	& \quad+ \tfrac{1}{2} \left( g_{i\ell}g_{jm} (L_X\pi)^{\ell m} + (\Div_g X) \pi_{ij}- X_{k;m} \pi^{km} g_{ij} - g(X, J) g_{ij} \right) -\tfrac{1}{2} (X\odot J)_{ij}, \\
	&\quad \left. -\tfrac{1}{2} (L_Xg)^{ij} + \left(\tfrac{2}{n-1} (\mbox{tr}_g \pi ) g^{ij}- 2 \pi^{ij}  \right) f\right),
\end{split}
 \end{align}
where  $L_g^*f = -(\Delta_g f)g + \textup{Hess}_g f - f \textup{Ric}(g)$ and the indices are raised or lowered by $g$.

 \begin{lemma}\label{lemma:adjoint-equations}
 Let $(f, X)$ solve   $(D\overline{\Phi}_{(g,\pi)} )^*(f, X)=(h,w)$. Then  $(f, X)$ satisfies the following Hessian type equations:
 \begin{align*}
	h_{ij} - \tfrac{1}{n-1}(\mathrm{tr}_g h)g_{ij} &= f_{;ij}  +\left[ -R_{ij} + \tfrac{2}{n-1} (\mathrm{tr}_g \pi) \pi_{ij} - 2 \pi_{ik} \pi^k_j + \tfrac{1}{n-1} \left(R_g- \tfrac{2}{n-1} (\mathrm{tr}_g \pi)^2 + 2|\pi|^2\right)g_{ij}\right] f\\
	&\quad + \tfrac{1}{2} \left( g_{i\ell}g_{jm} (L_X\pi)^{\ell m} + (\Div_g X) \pi_{ij}- X_{k;m} \pi^{km} g_{ij} - g(X, J) g_{ij} \right)-\tfrac{1}{2}(X\odot J)_{ij}\\
	&  \quad-\tfrac{1}{2(n-1)} \left( \mathrm{tr}_g (L_X \pi) +  (\Div_gX) (\mathrm{tr}_g \pi) - n  X_{k;m}\pi^{km} -  (n+1)g(X, J)\right) g_{ij}\\
	 - w_{ij;k}-w_{ki;j} + w_{jk;i} &=   X_{i;jk}+\tfrac{1}{2} (R^{\ell}_{kji} + R^{\ell}_{ikj} + R^{\ell}_{ijk}) X_{\ell}\\
	&\quad-\left(\left(\tfrac{2}{n-1} (\mathrm{tr}_g \pi) g_{ij} -2 \pi_{ij} \right) f\right)_{;k}-\left(\left(\tfrac{2}{n-1} (\mathrm{tr}_g \pi) g_{ki} - 2\pi_{ki} \right) f\right)_{;j}\\
	&\quad +\left(\left(\tfrac{2}{n-1} (\mathrm{tr}_g \pi) g_{jk} - 2\pi_{jk} \right) f\right)_{;i},
\end{align*}
where the indices are raised and lowered by $g$. %By taking the trace, $(f, X)$ satisfies the following elliptic system
%\begin{align}\label{equation:elliptic}
%\begin{split}
%	-\tfrac{1}{n-1}\mathrm{tr}_g h &= \Delta_g f+\tfrac{1}{n-1}\left(R_g-\tfrac{2}{n-1} (\mathrm{tr}_g \pi)^2 + 2|\pi|_g^2\right) f \\
%	&\quad - \tfrac{1}{2(n-1)} \left[  \mathrm{tr}_g (L_X \pi) +  (\Div_gX)(\mathrm{tr}_g \pi)  - n \pi^{km} X_{k;m} - (n+1) g(X, J)\right]  \\
%	 -2\Div_g w+ d(\mathrm{tr}_g w) &= \Delta_g X + R^{\ell}_i X_\ell -\tfrac{2}{n-1} d(f\mathrm{tr}_g \pi ) + 4\Div_g(f\pi ).
%	 \end{split}
%\end{align}
 \end{lemma}
\begin{proof}
By taking the trace of the first component of $(D\overline{\Phi}_{(g,\pi)})^*(f, X)=(h,w)$, we obtain the Laplace equation for $f$. Using that equation to eliminate  the Laplace term in the first component of $(D\overline{\Phi}_{(g,\pi)} )^*(f, X)=(h,w)$ gives the Hessian equation for $f$.

By commuting the order of derivatives and the Ricci formula, 
\begin{align*}
	& (L_Xg)_{ij;k} +  (L_X g)_{ki;j}- (L_Xg)_{jk;i}  \\
 	&=( X_{i;jk}+ X_{i;kj} )+ (X_{j;ik}-X_{j;ki}) + (X_{k;ij}-X_{k;ji})\\
	 &= 2X_{i;jk}+( R^{\ell}_{kji} + R^\ell_{ikj} + R^\ell_{ijk})X_{\ell}
\end{align*}
where the sign convention for the Riemannian curvature tensor is so that the Ricci tensor $R_{jk} = R^{\ell}_{\ell jk}$. Together with the equations for $L_Xg$ from $(D\overline{\Phi}_{(g,\pi)} )^*(f, X)=(h,w)$, it implies the Hessian equation of $X$.  

\end{proof}

The following proposition is a consequence of elliptic regularity.
\begin{proposition}\label{proposition:regularity}
Let $(M, g, \pi)$ be an initial data set with $(g-g_{\mathbb{E}},\pi) \in C^{2}_{-q}\times C^{1}_{-1-q}$. Let $a> 1$ and $\delta\in (0, q]$. Suppose $(f, X)\in L^a_{-\delta}$  and $(h, w)\in C^{0}_{-2-q}\times C^{1}_{-1-q}$ so that $(D\overline{\Phi}_{(g,\pi)})^* (f, X) = (h, w)$ weakly.  Then $(f,X)\in C^{2 }_{-q}$. Furthermore, if $(g-g_{\mathbb{E}},\pi) \in C^{2, \alpha}_{-q}\times C^{1, \alpha}_{-1-q}$, then $(f, X)\in C^{2,\alpha}_{-q}$. 
\end{proposition}
\begin{proof}
Fix an atlas on $M$ that consists of finitely may compact charts and the chart at infinity. If we express the equation $(D\overline{\Phi}_{(g,\pi)})^*(f,X) = (h,w)$ in a coordinate chart, there are coefficient matrices $A, B$ so that 
\begin{align} \label{equation:hessian}
	(\nabla^2 f, \nabla^2 X) + A (\nabla f, \nabla X) + B (f, X) \in C^0_{-2-q}.
\end{align}
By the explicit expressions for $A, B$ in Lemma \ref{lemma:adjoint-equations} and the hypothesis $(g, \pi) \in C^2_{-q}\times C^1_{-1-q}$, the entries of $A$ and $B$ are in $C^1_{-1-q}$ and $C^0_{-2-q}$ respectively. Taking the (Euclidean) trace of \eqref{equation:hessian} and applying Corollary~\ref{corollary:Lp-estimate} to $\Delta_0 f$ and to each component of $\Delta_0 X$ gives $(f, X) \in W^{2,a}_{-\delta}$. 

Next, we use a bootstrap argument  to show that $(f, X)\in W^{2,s}_{-\delta}$ for some $s>n$ and thus $(f, X)\in C^{1,1-\frac{n}{s}}_{-\delta}$  by the weighted Sobolev inequality.  We prove only the case that $a<n/2$, as the other case $a\in [n/2, n]$ can be treated similarly.  By the weighted Sobolev inequality, the initial regularity $(f, X)\in W^{2,a}_{-\delta}$ implies that
\[
	(\nabla f, \nabla X) \in L^{\frac{na}{n-a}}_{-1-\delta} \quad \mbox{and} \quad (f, X)\in L^{\frac{na}{n-2a} }_{-\delta}.
\]
Substituting those into the trace of \eqref{equation:hessian} implies that $(\Delta_0 f, \Delta_0 X)\in L^{\frac{na}{n-a}}_{-2-\delta}$. Applying Corollary~\ref{corollary:Lp-estimate}  gives $(f, X) \in W^{2,\frac{na}{n-a}}_{-\delta}$.  We repeat the argument using the improved regularity each time. After finitely many steps, we can conclude that $(f, X) \in W^{2,s}_{-\delta}$ for some $s>n$.

So far, we have only used the trace of \eqref{equation:hessian}. To show the $C^2$-regularity and to improve the fall-off rate, we employ the full Hessian equation \eqref{equation:hessian}. Using the fall-off rates of the coefficients $A, B$, we obtain $(\nabla^2 f, \nabla^2 X)\in C^0_{-2-q}$ and hence $(f, X) \in C^2_{-q}$. 

If $(g-g_{\mathbb{E}},\pi) \in C^{2, \alpha}_{-q}\times C^{1, \alpha}_{-1-q}$, then the coefficients matrices $A\in C^{1,\alpha}_{-1-q}$ and $B\in C^{0,\alpha}_{-2-q}$ and hence $(f, X) \in C^{2,\alpha}_{-q}$ by standard elliptic regularity.

\end{proof}
%\begin{remark}
%Since $(f, X)\in C^{2,\alpha}_{-q}$ implies $(f,X)\in L^p_{-q+\epsilon}$ for any $\epsilon>0$, we also have the following estimate: for each $\epsilon>0$, there is $C$ such that
%\begin{align*}
%	\| (f, X)\|_{C^{2,\alpha}_{-q+\epsilon}} &\le C\left(\| (f, X)\|_{L^a_{-q+\epsilon}} + \|(h,w)\|_{C^{0,\alpha}_{-2-q+\epsilon}\times C^{1,\alpha}_{-1-q+\epsilon}}\right).
%\end{align*}
%\end{remark}

\section{Proof of Lemma~\ref{lemma:surjective}}\label{section:surjective}

We now prove the surjectivity lemma. In contrast with the proof of~\cite[Proposition 3.1]{Corvino-Schoen:2006}, ours  uses an elementary ODE argument in place of their unique continuation argument for elliptic systems.

\begin{surjective lemma}
Let $(M, g, \pi)$ be an initial data set with $(g- g_\mathbb{E}, \pi) \in C^{2}_{-q} \times C^{1}_{-1-q}$. The modified constraint map $\overline{\Phi}_{(g,\pi)}:  \mathscr{M}^{2,p}_{-q}\times W^{1,p}_{-1-q}\too L^p_{-2-q}$ is smooth, and $D\overline{\Phi}_{(g,\pi)}: W^{2,p}_{-q}\times W^{1,p}_{-1-q}\too L^p_{-2-q}$ is surjective.  
\end{surjective lemma}
\begin{proof}
Arguing as in \cite[Corollary 3.2]{Bartnik:2005}, we see that $\overline{\Phi}$ is smooth. Since $D\overline{\Phi}_{(g,\pi)}$ has closed range (see, for example,  \cite[p. 110]{Eichmair-Huang-Lee-Schoen:2016}), it suffices to show that the adjoint map $(D\overline{\Phi}_{(g,\pi)})^*$  has trivial kernel to obtain the desired surjectivity of $D\overline{\Phi}_{(g,\pi)}$. That is, if $(f, X) \in (L^p_{-2-q})^* = L^{\frac{p}{p-1}}_{-n+2+q}$ satisfies $(D\overline{\Phi}_{(g,\pi)})^*(f, X)=0$ weakly, then $(f, X)$ is identically zero. To simplify the notation, we denote $a=\frac{p}{p-1}$ and $\delta =n-2-q$. 

Let $(f, X) \in L^a_{-\delta}$ weakly solve the equation $(D\overline{\Phi}_{(g,\pi)})^*(f, X)=0$. 
Applying Proposition~\ref{proposition:regularity} with $(h, w)=0$, we have $(f,X)\in C^{2 }_{-q}$. We would like to show that that $(f,X)$ has infinite order decay, that is, $(f, X) \in C^2_{-N}$ for any $N>0$. 

Recall that Lemma~\ref{lemma:adjoint-equations}  implies  that $(f, X)$ weakly solves an over-determined elliptic system
\begin{align} \label{equation:hessian2}
	(\nabla^2 f, \nabla^2 X) = A (\nabla f, \nabla X) + B (f, X),
\end{align}
where the entries of $A, B$ are in  $C^1_{-1-q}, C^{0}_{-2-q}$, respectively.  
Substituting the fall-off rate of $(f, X)$ into the right hand side of \eqref{equation:hessian2} and using the fall-off rate of the coefficient matrices $A, B$, we obtain $(\nabla^2 f, \nabla^2 X)\in C^0_{-2-q-\delta}$ and hence $(f, X) \in C^2_{-q-\delta}$. Bootstrapping with an improved fall-off rate each time into \eqref{equation:hessian} ultimately yields decay as fast as we like.

 To finish the proof, we  restrict \eqref{equation:hessian} along a radial geodesic $\gamma(t)$ to obtain an ODE system for $Z(t) = (f(\gamma(t)), X(\gamma(t))$: 
\[
	Z''(t) = A(t) Z'(t) + B(t) Z(t)
\]		
with $|A(t)|+t|B(t)| \le Ct^{-1-q}$ and $|Z(t)|+ |Z'(t)|\le C_N t^{-N}$ for  all $t>1$ and for all $N$. An elementary ODE argument then shows that $Z$ is identically zero. See, for example,  \cite[Lemma B.3]{Huang-Martin-Miao:2018} (note the missing hypothesis $|Z'(t)|\le C_N t^{-N}$ in the statement of that lemma). By varying geodesics, we show that $(f, X)$ is identically zero on $M$.

\end{proof}

\section{The method of Lagrange multipliers}\label{section:Lagrange}

Our variational approach  relies on the Lagrange multiplier theorem for constrained minimization. The version presented here suits better a local extreme problem, as opposed to another standard version for critical points (\textit{e.g.} the one used by Bartnik in  \cite[Theorem 6.3]{Bartnik:2005}). %The advantage  is that the variations can be taken in the natural weighted Sobolev spaces, and we do not need to establish a splitting theorem that  the kernel of the linearized operator is complemented. 
%We remark that nevertheless the constraint map can be a submersion between the appropriate Banach spaces. Bartnik's framework is three-dimensional, and a Hilbert space structure can be imposed for which the splitting holds trivially. For higher dimensions, P.~Chru\'sciel and E.~Delay show the constraint map is a submersion between more delicate Banach spaces~\cite{Chrusciel-Delay:2004}, which, however, requires the initial data sets to be at least $C^{4,\alpha}\times C^{4,\alpha}$ regular. It is related to  the well-known \emph{derivative loss} issue of their approach (see also \cite{Corvino-Schoen:2006}  and \cite{Chrusciel-Delay:2003}). 
The proof is simple and can be found in {\cite[Section 9.3]{Luenberger:1969}}. Since it is an important ingredient of the main result, we include the proof for completeness.

\begin{theorem}\label{theorem:Lagrange}
Let $X, Y$ be Banach spaces, and let $U$ be an open subset of $X$. Let $f: U\too \mathbb{R}$ and $h:U\too Y$ be  $C^1$. Suppose  $f$ has  a local extreme (minimum or maximum) at $x_0\in U$ subject to the constraint $h(x)=0$,  and suppose $Dh(x_0)$ is surjective. Then
\begin{enumerate}
\item $Df(x_0) (v) = 0$ for all $v\in \mathrm{ker}(Dh(x_0))$. \label{equation:kernel}\label{item:annihilator}
\item There is $\lambda\in Y^*$ such that $Df(x_0)  = \lambda (Dh(x_0))$, i.e. for all $v\in X$, 
\[
	Df(x_0) (v) =  \lambda (Dh(x_0) (v)).
\]
\end{enumerate}
\end{theorem}
\begin{proof}
We may without loss of generality assume that $f(x_0)$ is a local minimum subject to the constraint $h(x)=0$. Define a  $C^1$ map $T:U\too \mathbb{R}\times Y$ by 
\[
	T(x) = (f(x), h(x)). 
\]
We prove the first claim. Suppose on the contrary that there is $v\in \mathrm{ker}(Dh(x_0))$ so that $Df(x_0)(v)\neq 0$. It implies $DT(x_0) = (Df(x_0), Dh(x_0))$ is surjective because $Dh(x_0)$ is surjective.  By the Local Surjectivity Theorem~(\cite[Theorem 1, Section 9.2]{Luenberger:1969}), for any $\epsilon>0$, there exists $x\in U$ and  $\delta>0$ such that  $|x-x_0| <\epsilon$ and $T(x) = (f(x)-\delta, 0)$. This contradicts the assumption that $x_0$ is a local minimum of $f(x)$ subject to the constraint $h(x)=0$. 

The first claim says that $Df(x_0)$, as an element in the dual space $X^*$, lies in the annihilator subspace $(\mathrm{ker} Dh(x_0))^\perp$  of the dual space $X^*$ with respect to the natural pairing of $X$ and $X^*$. Because $Dh(x_0)$ has closed range, we have $(\mathrm{ker} Dh(x_0))^\perp = \mathrm{range} ((Dh(x_0))^*)$ (see \cite[Theorem 2, Section 6.6]{Luenberger:1969} for this fact). It implies there is $\lambda \in Y^*$ so that 
\[
	Df(x_0) = (Dh(x_0) )^*(\lambda). 
\]
By the definition of adjoint operators, for all $v\in X$, 
\[
	Df(x_0)(v) = (Dh(x_0) )^*(\lambda) (v) = \lambda (Dh(x_0)(v)).
\]
\end{proof}

\bibliographystyle{amsplain}
\bibliography{2019}

\end{document}